\documentclass[reqno,12pt]{amsart}
\usepackage{amsmath,amsfonts,amsthm,amscd,amssymb,graphicx,enumerate}

\newtheorem{theorem}{Theorem}[section]

\newtheorem{lemma}[theorem]{Lemma}

\newtheorem{proposition}[theorem]{Proposition}

\newtheorem{corollary}[theorem]{Corollary}

\newtheorem{example}[theorem]{Example}

\newtheorem{remark}[theorem]{Remark}
\newtheorem{definition}[theorem]{Definition}

\begin{document}
\title[Analytic Function Theory]{Analytic Function Theory for Operator-Valued Free Probability.}

\author[Williams, John D.]{John D.\ Williams}\let\thefootnote\relax\footnotetext{Supported in part by NSF grant DMS-1302713.}
\address{J.\ Williams, Department of Mathematics, Texas A\&M University,
College Station, TX 77843-3368, USA}
\email{jwilliams@math.tamu.edu} 
\let\thefootnote\relax

\begin{abstract}
It is a classical result in complex analysis that the class of functions that arise as the  Cauchy transform of probability measures may be characterized entirely in terms of their analytic and asymptotic properties.  Such transforms are a main object of study in non-commutative probability theory as the function theory encodes information on the probability measures and the various convolution operations.  In extending this theory to operator-valued free probability theory, the analogue of the Cauchy transform is a non-commutative function with domain equal to the non-commutative upper-half plane.  In this paper, we prove an analogous characterization of the Cauchy transforms, again, entirely in terms of their analytic and asymptotic behavior.  We further characterize those functions which arise as the Voiculescu transform of $\boxplus$-infinitely divisible operator-valued distributions.  As consequences of these results, we provide a characterization of infinite divisibility in terms of the domain of the relevant Voiculescu transform, provide a purely analytic definition of the semigroups of completely positive maps associated to infinitely divisible distributions and provide a Nevanlinna representation for non-commutative functions with the appropriate asymptotic behavior.
\end{abstract}
\maketitle

\section{Introduction}
Given a Borel probability measure $\mu$ on $\mathbb{R}$, the \textit{Cauchy transform} of this measure is the function
$$G_{\mu}(z) := \int_{\mathbb{R}} \frac{1}{z-t} d\mu(t) : \mathbb{C}^{+} \rightarrow \mathbb{C}^{-}.$$
The following classical result characterizes this class of functions in terms of their analytic and asymptotic properties, with immediate converse.
\begin{theorem}\label{scalar}
Let $g: \mathbb{C}^{+} \rightarrow \mathbb{C}^{-}$ denote an analytic function such that
$h(z):= g(1/z)$ has analytic continuation to a neighborhood of zero and $(iy)g(iy) \rightarrow 1$ as $y \uparrow \infty$.
Then, there exists a compactly supported Borel probability measure $\mu$ on $\mathbb{R}$ such that $g = G_{\mu}$.
\end{theorem}
We define the \textit{F-transform} of this measure as $$F_{\mu}(z) := \frac{1}{G_{\mu}(z)}.$$
The following theorem, due to Nevanlinna, provides both a classification and representation of the class of functions that arise as the reciprocal of a Cauchy transform.
\begin{theorem}\label{nevanlinna1}
Let $F : \mathbb{C}^{+} \rightarrow \mathbb{C}^{+}$ denote an analytic function.  The following are equivalent:
\begin{enumerate}
\item We have that $$H(z) := \frac{1}{F(1/z)}$$ extends to a neighborhood of $0$.  Moreover, $$\lim_{y\uparrow \infty} \frac{F(iy)}{iy} = 1. $$
\item  There exists an $\alpha \in \mathbb{R}$ and a compactly supported, positive  Borel measure $\rho$ on $\mathbb{R}$ such that $$ F(z) = \alpha + z + \int_{\mathbb{R}} \frac{1 + tz}{t - z}d\rho(t) $$
\item There exists a compactly supported probability measure $\mu$ on $\mathbb{R}$ such that $F = F_{\mu}$ on $\mathbb{C}^{+}.$
\end{enumerate}
\end{theorem}  It is the purpose of this paper to prove analogues of these theorems for non-commutative, analytic functions.

Non-commutative probability theory was developed by Voiculescu at the end of the $20$th century in order to study free product phenomenon in operator theory. The theory has developed along similar lines as classical probability theory with a convolution operation, $\boxplus$ , serving as an analogue of classical convolution.  The Cauchy transform is used to develop a linearizing transform  similar to the characteristic function in the classical case.  In particular, we may define the \textit{R-transform}, as a function with domain equal to  a subset of $\mathbb{C}^{-}$ through the equation
$$\mathcal{R}_{\mu}(z) := G_{\mu}^{\langle -1 \rangle} (z) - \frac{1}{z}.$$
This equation satisfies the equality $$ \mathcal{R}_{\mu \boxplus \nu} = \mathcal{R}_{\mu} + \mathcal{R}_{\nu}$$
on a subset of the complex upper half plane where each of these functions is defined.  Thus, in studying the Cauchy transform one may recover the distribution and study this convolution operation, hence its importance as an object of study.

In extending this theory to amalgamated free products, Voiculescu developed operator-valued free probability theory, with the main object of study being the so-called $\mathcal{B}$\textit{-valued distributions} arising from the conditional expectations of C$^{\ast}$-algebras on to C$^{\ast}$-subalgebras (see \cite{V3}).  The function theory associated to the operator-valued case was shown in \cite{V2} to be a particular case of Taylor's non-commutative function theory which was initially developed in \cite{TT1} and \cite{TT2}.  Thus, the relevant transforms in the operator-valued setting are non-commutative functions and it is the primary aim of this paper to prove a version of Theorem \ref{scalar} in this setting.  This result, along with other recent works such as \cite{PV} and \cite{BS2} should provide the basis for a robust function theoretic approach to the study of $\mathcal{B}$-valued distributions.

The main result in this paper is Theorem \ref{Main_theorem} which is a near exact analogue of Theorem \ref{scalar} in this more complicated, operator-valued setting.  The one aspect of this result where the analogy with the scalar-valued case breaks down is that \ref{scalar} only requires that the appropriate asymptotic behavior hold in a \textit{single direction} whereas the operator-valued result does not permit such weakened hypotheses.  This is addressed at length in Example \ref{example}.  These techniques are also brought to bear on the linearizing  transforms and an analogue of Theorem \ref{Main_theorem} is proven for the Voiculescu transforms of infinitely divisible distributions  in Theorem \ref{VT}.  

There are several important corollaries to these results.  First, in Proposition \ref{extension}, we show that a distribution $\mu$ is infinitely divisible if and only if its Voiculescu transform extends to $H^{+}(\mathcal{B})$.  This is in exact analogy with the scalar-valued case (see \cite{BV}, section 5).
Second, in Corollary \ref{semigroups} we show, using purely analytic techniques, that to each infinitely divisible distribution, one may associate a semigroup of divisors indexed by the completely positive self-maps of the algebra $\mathcal{B}$.  This was 
shown in \cite{ABFN} and \cite{Dima2} through Fock space constructions and we recover this result through function theoretic methods.  Lastly, in Corollary \ref{nevanlinna}, we are able to show that the family of non-commutative functions that arise as the $F$-transforms of $\mathcal{B}$-valued distributions may also be characterized by their analytic and asymptotic properties.  This result, combined with \ref{VT} and the main result in \cite{PV}, allows us to prove Corollary \ref{nevanlinna} which is  a direct analogue of \ref{nevanlinna1} (although we note here that the full scope of Nevanlinna's result extends to measures which are not necessarily compactly supported).  As this representation is invaluable in the study of scalar-valued free probability,  \ref{nevanlinna} should be an important tool in the continuing development of operator-valued free probability.

This paper is organized as follows.  In section \ref{preliminaries}, the preliminary material associated with vector valued analytic functions (\ref{analytic}), operator-valued free probability (\ref{Op_Prelim}), non-commutative functions (\ref{NC_prelim}) and Cauchy transforms (\ref{transforms}) is presented.   In section \ref{main_result}, we prove our main result, the classification of the Cauchy transforms in this non-commutative setting.  In section \ref{transform_result}, we prove an analogous characterization of the linearizing transforms associated to $\mathcal{B}$-valued distributions.  In \ref{consequences}, we derive some of the many consequences of this result, including Nevanlinna type representations for certain classes of non-commutative functions and defining semigroups of completely positive maps associated to each infinitely divisible distribution.

\textit{Acknowledgements:}  The author would like to thank Michael Anshelevich for posing this question and making himself available for extended discussions and Hari Bercovici for providing excellent ideas and advice.  The final portions of this project were completed during the Fields Institute's focus program on non-commutative probability theory and the author would like to thank the staff for providing an excellent environment in which to work.  We would also like to thank the referee for a thorough review and helpful suggestions.

\section{Preliminaries}\label{preliminaries}
\subsection{Vector Valued Analytic Functions}\label{analytic}

We shall refer to \cite{Hille}, sections $3$ and $26$ and the references therein as a blanket reference for this material.
We single out the references \cite{zorn1}, \cite{zorn2} and \cite{zorn3} as especially applicable to this section although many of these results may be considered classical in nature and free from specific reference.
We also note that many of these results hold in a greater level of generality but are dampened for clarity as our requirements are more modest.

Let $\mathcal{X}$ and $\mathcal{D}$ denote complex Banach spaces and $\mathcal{U} \subset \mathcal{X}$ an open set.  Consider a function
$f: \mathcal{U} \mapsto \mathcal{D}$.

\begin{definition}
The function $f$ is said to be G\^{a}teaux differentiable if for every $x \in \mathcal{U}$ and $h \in \mathcal{X}$,
the quotient $$\frac{f(x + \zeta h) - f(x)}{\zeta} , $$ which is defined for $\zeta \in \mathbb{C}$ small enough, tends to a unique limit as  $\zeta \rightarrow 0$.
In symbols, this limit shall be referred to as $\delta f(x;h) = \delta^{h}_{x} f$, the first variation of $f(x)$ with increment $h$.
\end{definition}
In analogy with complex analytic functions, G\^{a}teaux differentiable functions on an open set $\mathcal{U}$ have derivatives of all orders.
Further, the $n$th variation of $f(x)$ with increment $h$ (in symbols $\delta^{n}f(x;h)$) is a homogeneous function of degree $n$ in $h$.  That is, given $\alpha \in \mathbb{C}$,
we have that $\delta^{n}f(x;\alpha h) = \alpha^{n}\delta^{n}f(x;h)$.  In particular, the first derivative is a linear function in $h$.

\begin{definition}
The function $f$ is said to be Fr\'echet differentiable if $\delta f(x;h)$ is a bounded, homogeneous function of degree one in $h$
and $$ \lim_{\| h \| \rightarrow 0} \frac{1}{\|h\|} \|f(x+h) - f(x) - \delta f(x;h)  \| = 0 $$
for all $x \in \mathcal{U}$.
\end{definition}

It is implicit in the definition that Fr\'echet differentiable functions are in fact G\^{a}teaux differentiable.  In fact, very robust converses exist as we shall see
later in the section.

\begin{definition}
The function $f$ is said to be locally bounded if for every $a \in \mathcal{U}$,
there exists an $r_{a} > 0$ and a finite $M(a) > 0$ such that
$\| x - a \| < r_{a}$ implies that $\|f(x) \| < M(a)$.
\end{definition}

\begin{definition}
The function $f$ is said to be analytic if it is locally bounded and G\^{a}teaux differentiable in $\mathcal{U}$.
\end{definition}

The following theorem justifies this
definition insofar as it shows that analytic functions are precisely the limits of power series.  It also provides a converse for the varying strengths of differentiability.

\begin{theorem}\label{analytic_theorem}
Assume that $f$ is analytic in $\mathcal{U}$, then it is continuous and Fr\'echet differentiable in $\mathcal{U}$.
For $x \in \mathcal{U}$, $\delta^{n} f(x;h)$ is a bounded homogeneous function of degree $n$ in $h$
and a locally bounded function of $x$.  To each $a \in \mathcal{U}$ there exists a $t_{a} > 0$
such that the Taylor expansion
$$ f(x+h) = \sum_{n=0}^{\infty} \frac{1}{n!} \delta^{n}f(x;h) $$
converges uniformly for $\| x - a \| < t_{a}$ and $\| h \| < t_{a}$
\end{theorem}

 Theorem 3.16.3 in \cite{Hille} and is a useful analogue of the classical
Cauchy estimates in complex analysis.
\begin{theorem}\label{Cauchy}
Let $f$ be G\^{a}teaux differentiable in $\mathcal{U}$ and assume that $\|f(x) \| \leq M$ for $x \in \mathcal{U}$.
Then $$\| \delta^{n}f(a;h) \| \leq Mn!$$
for $a + h \in \mathcal{U}$.
\end{theorem}
Further, theorem 3.17.17 in \cite{Hille} provides Lipschitz estimates for analytic functions.
Indeed, for an analytic function $f$ that is locally bounded by $M(a)$ in a neighborhood of radius $r_{a}$, we have that
\begin{equation}\label{lipschitz}
\| f(y) - f(x) \| \leq \frac{2M(a) \| x - y \|}{r_{a} - 2 \| x - y \|}
\end{equation}
These will come prove useful as we will utilize the Kantorovich theorem repeatedly and its invokation relies on such estimates.
We note that this is a very powerful theorem with detailed estimates on the convergence of Newton method approximation.
Our requirements are very modest so we will only state an abbreviated version of Kantorovich's results.
Throughout, we let $\overline{B_{r}(x)}$ denote the closed ball of radius $r$ about $x$ in a Banach space $X$.

\begin{theorem}\cite{Kant2}\label{kantorovich}
Let $F: D \subset X \rightarrow Y$ denote an analytic function where $X$ and $Y$ are assumed to be Banach spaces.
Let $D_{0} \subset D$ denote an open convex set.
Assume that $F'(x_{0}) := \delta F(x_{0}, \cdot)$ is invertible as an operator and the following conditions are satisfied:
\begin{enumerate}
\item\label{kant1} $\| F'(x_{0})^{-1} (F'(x) - F'(y) ) \| \leq K \| x - y \| $ for all $x,y \in D_{0}$.
\item\label{kant2}  For $\eta = \| F'(x_{0})^{-1}  \cdot F(x_{0}) \|$ , we have $ h = K\eta \leq 1/2 $.
\end{enumerate}
Let $$ t^{\ast} = \frac{2\eta}{1 + \sqrt{1 - 2h}} ; \ t^{\ast \ast} = \frac{1 + \sqrt{1 - 2h}}{K} ; \ x_{1} = x_{0} - F'(x_{0})^{-1} \cdot F(x_{0})$$
and assume that  $ B_{t^{\ast}}(x_{1}) \subset D_{0}$ .  Then $F$ has a  root  $ x^{\ast} \in B_{t^{\ast}}(x_{1})$ and this root is unique in the set $B_{t^{\ast \ast}}(x_{1})$.
\end{theorem}

\subsection{Operator-Valued Free Probability}\label{Op_Prelim}

Free probability theory lies at the intersection of operator algebras, operator theory and probability theory.  We refer to \cite{KR} as an introductory text to the study of operator algebras and \cite{Vern} as an introduction to operator theory.

Let $\mathcal{A}$ and $\mathcal{B}$ denote $C^{\ast}$-algebras with $\mathcal{B} \subset \mathcal{A}$.
We call a \textit{conditional expectation} a positive, unital,  linear map $E: \mathcal{A} \mapsto \mathcal{B}$
with the property that $E(bab') = bE(a)b'$ for all $b,b' \in \mathcal{B}$ and $a \in \mathcal{A}$ (this property is called \textit{B-bimodularity}.)
  We shall refer to the triple $(\mathcal{A},E,\mathcal{B})$ as a \textit{$\mathcal{B}$-valued non-commutative probability space}.

Let $X$ denote a self-adjoint variable that is algebraically free from $\mathcal{B}$.  Let $\mathcal{B}\langle X \rangle$ be the $\ast$-algebra
of non-commutative polynomials over $\mathcal{B}$.  That is, the linear span of monomials of the form $b_{1}Xb_{2} \cdots Xb_{n+1}$ with $b_{i} \in \mathcal{B}$ for $i=1,\ldots ,n+1$.
The $\ast$ operation is defined on the monomials by  $b_{1}Xb_{2} \cdots Xb_{n+1} = b_{n+1}^{\ast}X \cdots b_{2}^{\ast}Xb_{1}^{\ast}$ and extended through linearity.

Given a triple  $(\mathcal{A},E,\mathcal{B})$ as above and an element $a\in A$, we define the \textit{$\mathcal{B}$-valued distribution} of $a$ to be the map
$\mu_{a} : \mathcal{B} \langle X\rangle \rightarrow \mathcal{B}$ defined by $\mu_{a}(P(X)) = E(P(a))$ where we abuse notation by letting $P(a)$ refer to the evaluation map on the non-commutative polynomial $P(X) \in \mathcal{B}\langle X \rangle$.

We define an abstract set of distributions $\Sigma$ by considering the set of all unital, positive $\mathcal{B}$-bimodular maps $\mu: \mathcal{B} \langle X \rangle \mapsto \mathcal{B}$ 
with the property that, for all $n \in \mathbb{N},$
\begin{equation}\label{CP}
 [\mu(P_{i}^{\ast}(X)P_{j}(X))]_{i,j=1}^{n} \geq 0
\end{equation}
in $M_{n}(\mathcal{B})$ for any family of elements $P_{i}(X) \in \mathcal{B} \langle X\rangle$ with $i=1,\ldots,n$.
We consider a subset $\Sigma_{0} \subset \Sigma$  by introducing the additional property that there exists $M > 0$
such that 
\begin{equation}\label{bound}
\| \mu(b_{1}Xb_{2} \cdots Xb_{n+1}) \| \leq \| b_{1} \| \cdots \|b_{n+1} \| M^{n}
\end{equation}
for all $b_{i} \in \mathcal{B}$ with $i= 1,\ldots , n$.

We say that a distribution $\mu \in \Sigma_{0}$  is $\boxplus$-\textit{infinitely divisible} if, for every $n \in \mathbb{N}$, there exists a distribution $\mu_{n} \in \Sigma_{0}$ such that
$$\mu = \underbrace{\mu_{n} \boxplus \mu_{n} \boxplus \cdots \boxplus \mu_{n}}_{n \ times}.$$
It was shown in \cite{Speichop}, Proposition 4.5.3,    that to each such distribution one may associate a semigroup $\{ \mu_{t} \}_{t \in \mathbb{R}}$ such that $\mu_{1} = \mu$ and
$\mu_{t} \boxplus \mu_{s} = \mu_{t + s}$.  Moreover, in \cite{ABFN} and \cite{Dima2}, this result is extended to semigroups indexed by completely positive self maps of $\mathcal{B}$.  In Corollary \ref{semigroups} we reprove this result by with an analytic methodology,  with the main tool being Theorem \ref{VT} and its 
analytic characterization of these distributions.

The following theorem makes clear why we emphasize the set $\Sigma_{0}$ and was proven in Proposition 2.2 of \cite{PV} (with a tracial variation of this characterization being proven in \cite{Wil3}).
\begin{theorem}\label{PV}
Let $\mu \in \Sigma$.
  Then $\mu \in \Sigma_{0}$ if an only if there exists a $\mathcal{B}$-valued non-commutative probability space $(\mathcal{A},E,\mathcal{B})$
and an element $a \in \mathcal{A}$ such that $\mu = \mu_{a}$.
\end{theorem}

We close this subsection by weakening the defining conditions of $\Sigma_{0}$.
\begin{proposition}\label{sigma}
 The $\mathcal{B}$-bimodularity assumption in the definition of $\Sigma_{0}$
may be weakened to $\mu|_{\mathcal{B}} = Id$.
\end{proposition}
\begin{proof}

Assume that $\mu$ satisfies the weakened assumptions stated in the hypothesis.
We first claim that
\begin{equation}\label{boundB}
\mu(P(X)^{\ast}  b^{\ast}b P(X)) \leq    \| b^{\ast}b \| \mu(P(X)^{\ast}P(X)) 
\end{equation}
for all $b \in \mathcal{B}$ and $P(X) \in \mathcal{B}\langle X \rangle$ (note that since $\mathcal{B}\langle X \rangle$ only has a $\ast$-algebra structure this does require argument) .
Indeed, since $\mathcal{B}$ is a C$^{\ast}$-algebra
$0 \leq \| b^{\ast}b \| - b^{\ast}b = c^{\ast}c$ for some $c \in \mathcal{B}$.
Thus, $$0 \leq \mu ((cP(X))^{\ast}cP(X)) = \mu(P(X)^{\ast}( \| b^{\ast}b \| - b^{\ast}b)P(X))$$
proving our claim.

We next claim that \begin{equation}\label{boundX}
\mu(Q(X)^{\ast}X^{2}Q(X)) \leq 4M^{2}\mu(Q(X)^{\ast}Q(X)) 
\end{equation}
for all $Q(X) \in \mathcal{B}\langle X \rangle$.
To do so, we must define a new $\ast$-algebra that will allow us to deal with infinite series arising from the non-commutative polynomials (the proof is similar to an analogous claim in \cite{PV}).


Let $M$ denote the constant arising from property \ref{bound}.
Fix $n \in \mathbb{N}$ and consider the tensor algebra $M_{n}(\mathcal{B})_{\otimes} := \oplus_{k=1}^{\infty} \otimes^{k} M_{n}(\mathcal{B})$.  The addition operation is entry-wise.
Multiplication is defined for monomials by $$b_{1} \otimes b_{2} \otimes \cdots \otimes b_{\ell + 1} \cdot c_{1} \otimes c_{2} \otimes \cdots \otimes c_{p} =b_{1} \otimes b_{2} \otimes \cdots \otimes b_{\ell + 1}  c_{1} \otimes c_{2} \otimes \cdots \otimes c_{p} $$
and the space is given the usual $\ast$-operation (note that this space is isomorphic to $\mathcal{B}\langle X \rangle$ as a $\ast$-algebra).
We endow this space with a norm .  Given a homogeneous degree $k$ polynomial, $P(X) \in \otimes^{k} M_{n}(\mathcal{B})$,  we refer to $|P(X)|$ as the infimum
of  $\sum_{i=1}^{n} \| b_{1}^{(i)} \| \| b_{2}^{(i)}\| \cdots \|b_{k}^{(i)}\| M^{k - 1}$
where the infimum is taken over the distinct sums satisfying $P(X) = \sum_{i=1}^{n} b_{1}^{(i)}\otimes b_{2}^{(i)} \otimes \cdots \otimes b_{k}^{(i)}$ .
Given an arbitrary element $P(X) = \sum_{k=1}^{N} P_{k}(X)$ with each $P_{k}(X) \in \otimes^{k} M_{n}(\mathcal{B})$ for $k = 1 , \ldots , N$,
we let $\| P(X)\| = \sum_{k=1}^{N} | P_{k}(X)|$.  This is indeed a norm (the proof is similar to showing the same for the  minimal tensor product norm). The closure of $M_{n}(\mathcal{B})_{\otimes}$  forms a Banach $\ast$-algebra
with respect to this norm.  It follows immediately from \ref{bound} that $\mu \otimes 1_{n} : M_{n}(\mathcal{B})_{\otimes} \rightarrow M_{n}(\mathcal{B})$ is a linear contraction.
Thus, we may extend $\mu$ to the norm closure of this space and we refer to this closure as $\tilde{\mu}$.

%

We will first focus on the $n=1$ case to prove \ref{boundX}.
For $k \in \mathbb{N}$ consider the following for $P(X) := \sum_{k=1}^{\infty} P_{k}(X) $:
\begin{equation}
P_{k}(X) := \frac{X^{2k} (2k)!}{(1-2k)(k!)^{2}(4M)^{2k}} 
\end{equation}
\begin{equation}\label{merter}
\|   P_{k}(X) \|  \leq  \frac{1}{4^{k}} \end{equation}
We claim that $P(X) \in \overline{\mathcal{B}_{\otimes}}$ and that 
\begin{equation}\label{squareroot}  P(X)^{2} = 1 - \frac{X^{2}}{4M^{2}}.\end{equation}
The first claim follows easily from our estimates.
Regarding \ref{squareroot}, let $$y = - \frac{X^{2}}{4M^{2}}.$$
Given that 
$$ P(X) = 1 + \frac{1}{2}y -  \frac{1}{8}y^{2} + \frac{1}{16}y^{3} - \frac{5}{128}y^{4} + \frac{7}{256}y^{5} \pm \cdots$$
we may take the Cauchy product ($\ast$) of this series with itself and observe that, treated as a purely formal power series,
$$ P(X) \ast P(X) = 1 + y$$
as all of the higher order terms cancel.  Thus, we need only show that the the Cauchy product is convergent (the proof is essentially that of Merten's classic result but we repeat it due to the delicate nature of our setting).
Towards this end, let $$S_{n}(x) = \sum_{k=1}^{n} P_{k}(X) ; \ c_{n}(X) = \sum_{i=1}^{n} P_{i}(X)P_{n-i}(X) ; \ C_{n}(X) = \sum_{i=1}^{n} c_{i}(X).$$
Note that
$$ C_{n}(X) = \sum_{i=1}^{n} P_{i}(X) S_{n-1}(X) = \sum_{i=1}^{n}  P_{i}(X) (S_{n-i}(X) - P(X)) + S_{n}(X)P(X)$$
Thus, \begin{align} \| C_{n}(X) & - P(X)^2\|   =  \left\| \sum_{i=1}^{n}  P_{i}(X) (S_{n-i}(X) - P(X)) + [S_{n}(X) - P(X)] P(X) \right\| \nonumber \\\label{merter2}
 & \leq   \left\| \sum_{i=1}^{N}  P_{i}(X) (S_{n-i}(X) - P(X))  \right\|  + \left\| \sum_{i=N}^{n}  P_{i}(X) (S_{n-i}(X) - P(X))  \right\| \nonumber \\  & + \| S_{n}(X) - P(X)\| \| P(X) \|
 \end{align}
For each of these terms, as this is a Banach algebra, the norm of the  product is dominated by the product of the norms.
Regarding the second term, we have that $\| S_{\ell}(X) - P(X) \|$ converges to $0$ as $\ell \uparrow \infty$ so is bounded over $\ell$.
By \ref{merter}, term $2$ is the tail of a convergent series so that it can be made smaller than $\epsilon > 0$ for $N$ large enough.
Fixing $N$, the first term can be made smaller than $\epsilon$ for $n$ large enough since $\|S_{n-i}(X) - P(X) \|$ is arbitrarily small for $i = 1 , \ldots , N$.  The third term is similarly smaller than $\epsilon$ for $n$ large.  This implies that
$C_{n}(X)$ is convergent and, therefore, 
$$P(X)^{2} = 1 -   \frac{X^{2}}{4M^{2}} \in  \overline{\mathcal{B}_{\otimes}} .$$ 
As a result of this equality, we have 
\begin{align*} 
0 \leq \tilde{\mu}(Q(X)^{\ast}P(X)^{\ast}P(X)Q(X)) & = \mu \left( Q(X)^{\ast} \left[1 - \frac{X^{2}}{4M^{2}} \right]Q(X) \right) \\
 & = \mu(Q(X)^{\ast}Q(X)) - \frac{\mu(Q(X)^{\ast}X^{2}Q(X))}{4M^{2}}.
\end{align*}
This inequality implies \ref{boundX}.

Now, let $P(X) = b_{1} X b_{2} \cdots X b_{\ell + 1}$ and $e \in \mathcal{B}$ a projection.  Observe that through inductive applications of \ref{boundB} and \ref{boundX} and the assumption that $\mu|_{\mathcal{B}} = Id$, we have
$$0 \leq \mu(e P^{\ast}(X) P(X) e) \leq (2M)^{2\ell} \| b_{1} \|^{2} \cdots  \| b_{\ell + 1} \|^{2} \mu(e) =  (2M)^{2\ell} \| b_{1} \|^{2} \cdots  \| b_{\ell + 1} \|^{2} e.$$
Thus, we have 
\begin{equation}\label{biproj1}
\mu(eP^{\ast}(X)P(X)e) = e\mu(eP^{\ast}(X)P(X)e)e.
\end{equation}

We next claim that 
\begin{equation}\label{biproj2}
(1-e)\mu(eP(X)) = 0
\end{equation}
for a fixed monomial $P(X) \in \mathcal{B}\langle X \rangle$.
Indeed, the following inequalities follow from \ref{CP}:
 \begin{align*}  0 & \leq 
 \left( \begin{array}  {cc}
    1 & 0 \\
 0 & 1-e
   \end{array} \right)
\mu \otimes 1_{2}\left[ \left( \begin{array}  {cc}
   1 & P^\ast(X)e \\
 eP(X) &  eP^{\ast}(X)P(X)e
   \end{array} \right) \right] 
 \left( \begin{array}  {cc}
    1 & 0 \\
 0 & 1-e
   \end{array} \right) \\
& =  \left( \begin{array}  {cc}
    1 & 0 \\
 0 & 1-e
   \end{array} \right)
 \left( \begin{array}  {cc}
   1 & \mu(P^\ast(X)e) \\
 \mu( eP(X)) & \mu( eP^{\ast}(X)P(X)e)
   \end{array} \right)  
 \left( \begin{array}  {cc}
    1 & 0 \\
 0 & 1-e
   \end{array} \right) \\
& =  \left( \begin{array}  {cc}
   1 & \mu(P^\ast(X)e)(1-e) \\
 (1-e)\mu( eP(X)) & (1-e)\mu( eP^{\ast}(X)P(X)e)(1-e)
   \end{array} \right)  
\end{align*}
Now, it is shown in chapter $3$, Exercise $ 3.2 (i) $ of \cite{Vern} that the positivity of the indicated matrix implies that, for any representation of
the algebra $\mathcal{B}$ on Hilbert space $\mathcal{H}$, we have that for $\eta ,  \xi \in \mathcal{H}$,
$$ |\langle \mu(P^{\ast}(X)e)(1-e)\eta , \xi \rangle|^{2} \leq  \langle  \eta , \eta \rangle \langle (1-e)\mu( eP^{\ast}(X)P(X)e)(1-e) \xi , \xi \rangle  $$
and it follows from \ref{biproj1} that the right hand side is equal to zero.  This implies \ref{biproj2}.

Let $P(X) \in \mathcal{B}\langle X \rangle$ and $e,f \in \mathcal{B}$ be projections.
Consider the following two equalities which follow from \ref{biproj2}:
\begin{align*}
& \mu(eP(X)f) \\  & = e \mu(eP(X)f)f + (1-e) \mu(eP(X)f)f +  e \mu(eP(X)f) (1-f)  + (1-e) \mu(eP(X)f)(1-f) \\ & = e\mu(eP(X)f)f 
\end{align*}
\begin{align*}
& e \mu(P(X) ) f  \\ & =  e \mu(e P(X) f ) f  +  e \mu((1-e) P(X)f ) f  +  e \mu( e P(X) (1-f) ) f  +  e \mu((1-e)P(X)(1-f) ) f  \\ & =  e \mu(eP(X)f)f
\end{align*}
It follows that  $\mu(eP(X)f) =  e \mu(P(X) ) f $.  If we assume that  $\mathcal{B}$ is a W$^{\ast}$-algebra, we have that
$\mu(bP(X)b') = b \mu(P(X) ) b' $ for any $b,b' \in \mathcal{B}$, proving our proposition in this case.

We conclude our proof by extending this proposition to the case of C$^{\ast}$-algebras.  To do so, we will extend our map $\mu$
to a map $\mu^{\ast \ast} : \mathcal{B}^{\ast \ast} \langle X \rangle \rightarrow \mathcal{B}^{\ast \ast}$.
We begin with a basic observation in functional analysis.

Let $\mathcal{X}$ and $\mathcal{Y}$ denote normed spaces and consider $\mathcal{X} \otimes \mathcal{Y}$ endowed with the projective norm.
We claim that $\mathcal{X}^{\ast \ast} \otimes \mathcal{Y}^{\ast \ast}$ embeds canonically  into $(\mathcal{X} \otimes \mathcal{Y})^{\ast \ast}$.
We will do so by showing that $\mathcal{X}^{\ast \ast} \otimes \mathcal{Y}^{\ast \ast}$ is in duality with $(\mathcal{X} \otimes \mathcal{Y})^{\ast }$ in a canonical manner.
Indeed, consider an element $\phi \in (\mathcal{X} \otimes \mathcal{Y})^{\ast} \cong B(\mathcal{X},\mathcal{Y}^{\ast})$.
Note that we may extend this map to  $\phi^{\ast} : \mathcal{Y} ^{\ast \ast} \rightarrow \mathcal{X}^{\ast}$
by letting $$ \phi^{\ast}(\eta) := \eta \circ \phi.$$  Repeating this process, we may extend $\phi$
to an element $\phi^{\ast \ast} \in B(\mathcal{X}^{\ast \ast} , \mathcal{Y}^{\ast \ast \ast}) \cong (\mathcal{X}^{\ast \ast} \otimes \mathcal{Y}^{\ast \ast})^{\ast}$
and a basic calculation shows that the $\phi^{\ast \ast}|_{\hat{\mathcal{X}}} = \phi$ where $\hat{\mathcal{X}}$ is the canonical isometric embedding of $\mathcal{X}$ into $\mathcal{X}^{\ast \ast}$.
Given $x \in \mathcal{X}^{\ast \ast}$ and $y \in \mathcal{Y}^{\ast \ast}$, we define an operation whereby $x \otimes y \cdot \phi := \phi^{\ast \ast}(x \otimes y)$ so that $x \otimes y \in [(\mathcal{X} \otimes \mathcal{Y})^{\ast}]^{\ast}$,
proving our claim.


Returning to our normed tensor algebras $ M_{n}(\mathcal{B})_{\otimes}$ , it follows from our functional analysis observation that $$ \oplus_{k=1}^{\infty} \otimes^{k} (M_{n}(\mathcal{B})^{\ast \ast})  \subset M_{n}(\mathcal{B})_{\otimes}^{\ast \ast}. $$
We define an extension $\mu^{\ast} : M_{n}(\mathcal{B})^{\ast } \rightarrow  M_{n}(\mathcal{B})_{\otimes}^{\ast }  $ by letting $\mu^{\ast}(\phi) := \phi \circ \mu $ for all $\phi \in M_{n}(\mathcal{B})^{ \ast}$.
Repeating this process, we obtain a canonical extension $\mu^{\ast \ast} : M_{n}(\mathcal{B})_{\otimes}^{\ast \ast }  \rightarrow  M_{n}(\mathcal{B})^{\ast \ast }$.

Now, observe that convergence of a  norm bounded sequence $b^{(n)} \rightarrow b$ in the strong operator topology on $M_{n}(\mathcal{B})^{\ast \ast}$ implies that the same is true in the strong operator topology on $ M_{n}(\mathcal{B})_{\otimes}^{\ast \ast }$
(since a faithful representation of $M_{n}(\mathcal{B})_{\otimes}^{\ast \ast }$ restricts to a faithful representation of $M_{n}(\mathcal{B})^{\ast \ast}$).
Thus, given a two sequences $b_{i}^{(n)} \rightarrow b_{i}$ in the strong operator topology on $M_{n}(\mathcal{B})^{\ast \ast}$,
we note that $b_{1}^{(n)}\otimes b_{2}^{(n)}  \rightarrow b_{1} \otimes b_{2}$  in the strong operator topology on $M_{n}(\mathcal{B})_{\otimes}^{\ast \ast }$.
Indeed, this follows from multiplicativity of the strong topology on bounded sets and the fact that $b_{1}^{(n)}\otimes b_{2}^{(n)} = b_{1}^{(n)}(1 \otimes 1)  b_{2}^{(n)}$.
Continuing inductively, $b_{1}^{(n)} \otimes b_{2}^{(n)} \otimes\cdots \otimes b_{k }^{(n)} \rightarrow b_{1} \otimes b_{2} \otimes \cdots \otimes b_{k} $ whenever the norm bounded sequences $b_{i}^{(n)} \rightarrow b_{i}$
in the strong operator topology on $M_{n}(\mathcal{B})^{\ast \ast}$ for each $i = 1 , \ldots , k$.

We are now able to show that our extension $\mu^{\ast \ast}$  is equal to the identity when restricted to $\mathcal{B}^{\ast \ast}$ and  satisfies properties \ref{CP} and \ref{bound} .
Indeed, given $\phi \in M_{n}(\mathcal{B})^{ \ast}$, we have that $\phi \circ \mu^{\ast \ast}$ is a bounded linear functional on  $M_{n}(\mathcal{B})_{\otimes}^{\ast \ast}$
so that a sequence $P_{n}(X) \rightarrow P(X)$ strongly in $M_{n}(\mathcal{B})_{\otimes}^{\ast \ast}$ implies that $\mu^{\ast \ast}(P_{n}(X)) \rightarrow \mu^{\ast \ast}(P(X))$
in the weak$^{\ast}$ topology on  $M_{n}(\mathcal{B})^{\ast \ast}$.  As all three of the properties survive weak limits (due variously to continuity, the fact that the positive cone is weak$^{\ast}$ closed and the Banach-Alaoglu theorem),
the extension $\mu^{\ast \ast}: \mathcal{B}^{\ast \ast} \langle X \rangle \rightarrow \mathcal{B}^{\ast \ast}$ satisfies all of the hypotheses of our proposition and is defined on a non-commutative W$^{\ast}$-probability space.
By the previous argument, this implies that $\mu^{\ast \ast}$ is $\mathcal{B}^{\ast \ast}$-bimodular so that we may conclude that $\mu \in \Sigma_{0}$, completing our proof.
\end{proof}

\subsection{Non-commutative Function Theory}\label{NC_prelim}

Throughout this subsection, we will utilize the definitions and terminology found in \cite{KVV}.

Let $\mathcal{B} , \mathcal{A}$ denote a unital $C^{\ast}$-algebras. 
We define the \textit{noncommutative space over $\mathcal{B}$} as the set $\mathcal{B}_{nc} = \{ M_{n}(\mathcal{B}) \}_{n=1}^{\infty}$.
A \textit{non-commutative set} is a subset $\Omega \subset \mathcal{B}_{nc}$
that respects direct sums.  That is, for $X \in \Omega \cap M_{n}(\mathcal{B})$ and $Y \in \Omega \cap M_{p}(\mathcal{B})$ we have that $X \oplus Y \in \Omega \cap M_{n+p}(\mathcal{B})$.
We note that these definitions apply for more general $\mathcal{B}$ over any unital, commutative ring,  but we focus on the $C^{\ast}$-algebraic setting.

A \textit{non-commutative} function is a map $f : \Omega \rightarrow \mathcal{A}_{nc}$ with the following properties:
\begin{enumerate}
 \item $f(\Omega_{n}) \subset M_{n}(\mathcal{A})$
\item  $f$ respects direct sums : $f(X \oplus Y) = f(X) \oplus f(Y)$
\item $f$ respects similarities: For $X \in \Omega_{n}$ and $S \in M_{n}(\mathbb{C})$ invertible we have that
$$f(SXS^{-1}) = Sf(X)S^{-1} $$
provided that $SXS^{-1} \in \Omega_{n}$.
\end{enumerate}

We say that a non-commutative set $\Omega$ is \textit{right-admissible}
if for every $X \in \Omega_{n}$, $Y \in \Omega_{p}$ and $Z \in M_{n+p}(\mathcal{B})$,
there exists a $\lambda \in \mathbb{C}$ such that
$$ \left[ \begin{array}  {cc}
    X & \lambda Z \\
 0 & Y
   \end{array} \right] \in \Omega_{n+p}. $$

Let $X \in M_{n}(\mathcal{B})$, $Y \in M_{p}(\mathcal{B})$ and $Z \in M_{n \times p} (\mathcal{B})$ and fix $Z \in M_{n \times p} (\mathcal{B})$
such that
$$  \left[ \begin{array}  {cc}
    X & Z \\
 0 & Y
   \end{array} \right] = M_{n+p}(\mathcal{B}) $$
We define a differential calculus for these non-commutative functions by defining the \textit{right difference-differential operator}, $\Delta_{R}$, 
implicitly through the equation $$f\left( \left[ \begin{array}  {cc}
    X & Z \\
 0 & Y
   \end{array} \right] \right) = \left[ \begin{array}  {cc}
    f(X) & \Delta_{R}f(X,Y)(Z) \\
 0 & f(Y)
   \end{array} \right]$$ 
for $X \in M_{n}(\mathcal{B})$, $Y \in M_{p}(\mathcal{B})$ and $Z \in M_{n \times p} (\mathcal{B})$.
If we assume that $\Omega$ is right admissible, for fixed $X$ and $Y$, this map extends to a linear operator from $M_{n \times p} (\mathcal{B})$ to $M_{n \times p} (\mathcal{A})$.

This process may be iterated on increasing orders of $2 \times 2$ block triangular matrices to define the \textit{ $\ell$-th higher order difference-differential operator}, $\Delta^{\ell}_{R}$.
However, for a right admissible set, the calculation reduces to the following equality, which is proven in Theorem $3.11$ in \cite{KVV}:
for $X_{0} \in \Omega_{n_{0}} ,  \ldots , X_{\ell} \in \Omega_{n_{\ell}}$ and  
$B_{1} \in M_{n_{0} \times n_{1}}(\mathcal{B}) ,  \ldots , B_{\ell} \in M_{n_{\ell - 1} \times n_{\ell}}$
we have that
\begin{align}
& f\left( \left[ \begin{array}  {cccccc}
    X_{0} & B_{1} & 0 &  \cdots & 0 & 0 \\
0 & X_{1} & B_{2} & \cdots & 0 & 0 \\
 \ & \vdots & \ & \ & \vdots & \ \\
0 & 0 & 0 & \cdots & X_{\ell - 1} & B_{\ell}  \\
0 & 0 & 0 & \cdots & 0 & X_{\ell}
   \end{array} \right] \right) \nonumber \\ & = \left[ \begin{array}  {cccccc}
    f(X_{0}) & \Delta_{R}f(X_{0},X_{1})(B_{1}) & \ &  \cdots & \ & \Delta_{R}^{\ell}f(X_{0} , \ldots , X_{\ell})(B_{1}, \ldots , B_{\ell }) \\
0 & f(X_{1}) & \ & \cdots & \ & \Delta_{R}^{\ell -1}f(X_{1} , \ldots , X_{\ell})(B_{2}, \ldots , B_{\ell})  \\
 \ & \vdots & \ & \ & \vdots & \ \\
0 & 0 & 0 & \cdots & f(X_{\ell - 1}) & \Delta_{R}f(X_{\ell - 1},X_{\ell})(B_{\ell})  \\
0 & 0 & 0 & \cdots & 0 & f(X_{\ell}) 
   \end{array} \right]\label{matrix}  
\end{align}
It can be shown that $\Delta^{\ell}_{R}$ is linear in each of the variables $B_{i}$.  This will become very important in later sections as these operators
will define operators on $\otimes_{k} M_{n}(\mathcal{B})$.

Let $f = (f^{(n)})_{n=1}^{\infty}$.  We extend the definition for locally bounded to this non-commutative setting by saying that a non-commutative function $f$ is \textit{locally bounded in slices}
if $f|_{\Omega_{n}}$ is locally bounded in the sense of subsection \ref{analytic}. Given an element $Y \in \Omega_{1}$, the \textit{non-commutative unit ball} of radius $r$, $B_{nc}(Y,r)$,
is equal to $\{ X \in \Omega_{n}: \| X - \oplus^{n} Y\| < r \}$ (all of the spaces that we will be dealing with are open in the topology generated by these balls so we will not dwell on the finer points of this theory).
A function is \textit{uniformly locally bounded} at $Y \in \Omega_{1}$ if $f$ there exists $r , M > 0$ such that $\| f^{(n)}(X) \| < M$ for all $n \in \mathbb{N}$ and $X \in B_{nc}(Y,r) \cap \Omega_{n}$.
The function $f$ is \textit{uniformly analytic}
 if it is uniformly locally bounded and  G\^{a}teaux differentiable.  The following theorem unites the analytic function theory from subsection \ref{analytic} with non-commutative function theory.
We refer to section $7$ of \cite{KVV} for proof.


\begin{theorem}\label{difference}
Let a non-commutative function $f: \Omega \rightarrow \mathcal{A}_{nc}$ be locally bounded on slices.
Then
\begin{enumerate}
 \item $f$ is G\^{a}teaux differentiable.
\item  For every $n \in \mathbb{N}$, $Y \in \Omega_{n}$, $Z \in M_{n}(\mathcal{B})$ and each $N \in \mathbb{N}$,
$$\frac{1}{N!} \frac{d^{N}}{dt^{N}} f^{(n)}(Y + tZ) |_{t=0} = \Delta_{R}^{N} f^{(n)}(\underbrace{Y,\ldots,Y}_{(N+1)- times})(\underbrace{Z,\ldots,Z}_{N - times}) $$
\end{enumerate}
Moreover,  there exists $r_{Y,n} > 0$ such that 
$$ f^{(n)}(X) = \sum_{\ell = 0}^{\infty} \Delta_{R}^{\ell} f^{(n)}(Y, \ldots , Y)(X - Y , \ldots , X-Y)$$
for all $X \in \Omega_{n}$ such that $\| X-Y \| < r_{Y,n}$.  Additionally, if $f$ is uniformly locally bounded at $Y \in \Omega_{1}$
then there exists a fixed $r > 0$ such that 
$$ f^{(n)}(X) = \sum_{\ell = 0}^{\infty} \Delta_{R}^{\ell} f^{(n)}(\oplus^{n}Y, \ldots , \oplus^{n}Y)(X - \oplus^{n}Y , \ldots , X-\oplus^{n}Y)$$
for all $X \in \Omega_{n} \cap B_{nc}(Y,r)$.
\end{theorem}

The following proposition states that the differential and matricial structure of non-commutative analytic functions coincide.  This will be of crucial importance as we will need to recover linear operators from the differentials of the non-commutative functions and this result will be invoked to show that the recovered operators are well defined.
We also refer to section $7$ of \cite{KVV}  for proof.

\begin{proposition}\label{matrix_prop}
Let $f$ denote a non-commutative function with $f^{(1)}(0) = 0$ .
Let  $B_{1} , \ldots , B_{\ell} \in M_{n}(\mathcal{B})$ with $B_{p} = (b_{i,j}^{(p)})_{i,j = 1}^{n}$ for $p = 1, \ldots , \ell$.  Then
\begin{align*}
& \Delta_{\mathcal{R}}^{\ell}f^{(n)}(0 , \ldots , 0 ; B_{1} , \ldots , B_{\ell}) \\ 
= & \left(\sum_{k_{1} , \ldots , k_{\ell - 1}= 1}^{n} \Delta_{\mathcal{R}}^{\ell}f^{(1)} (0,\ldots , 0; b^{(1)}_{i,k_{1}} , b^{(2)}_{k_{1} , k_{2}} , \ldots , b^{(\ell)}_{k_{\ell - 1},j} )  \right)_{i,j=1}^{n} .
\end{align*}
\end{proposition}


\subsection{$\mathcal{B}$-valued Distributions and Their Transforms.}\label{transforms}

We refer to \cite{V3} for the basics in the function theory for operator valued free probability.  We refer to \cite{Speichop} for the combinatorial aspects of the subject.
Let $\mu , \nu \in \Sigma_{0}$. Let $$M_{n}^{+ , \epsilon}(\mathcal{B}) := \{ b \in M_{n}(\mathcal{B}) : \ \Im{(b)} > \epsilon 1_{n} \}; \  M_{n}^{+}(\mathcal{B}) := \cup_{\epsilon > 0} M_{n}^{+ , \epsilon}(\mathcal{B}) .$$
 We define the $\mathcal{B}$-valued \textit{Cauchy transform} $G_{\mu} = (G_{\mu}^{(n)})_{n=1}^{\infty}$ with
$$ G_{\mu}^{(n)}: = \mu \otimes 1_{n}[(b-X \otimes 1_{n})^{-1}] = \sum_{n=0}^{\infty} \mu((b^{-1}X)^{n}b^{-1}) : M_{n}^{+}(\mathcal{B}) \rightarrow M_{n}^{-}(\mathcal{B}) .$$
It was shown in the pioneering work of Voiculescu \cite{V2} that these are non-commutative functions
with domain equal to the \textit{non-commutative upper half plane}, defined as $H^{+}(\mathcal{B}) = \{ M_{n}^{+}(\mathcal{B}) \}_{n=1}^{\infty}$.  We also define the set $H^{+}_{\epsilon}(\mathcal{B}) = \{ M_{n}^{+,\epsilon}(\mathcal{B}) \}_{n=1}^{\infty}$
for fixed $\epsilon > 0$.
Lastly, it will often be easier to utilize the \textit{F-transform}, defined through the equality
$$F_{\mu} := G_{\mu}^{-1} : H^{+}(\mathcal{B}) \rightarrow H^{+}(\mathcal{B}).$$  We note here that the $F$-transform satisfies the inequality $$\Im{(F^{(n)}(b))} \geq \Im{(b)}$$ for all $b \in M_{n}^{+}(\mathcal{B})$ and refer to \cite{BPV} for proof of this fundamental fact.

We define the \textit{R-transform} to be the function $$\mathcal{R}^{(n)}_{\mu}(b) := (G^{(n)}_{\mu})^{\langle -1 \rangle}(b) - b^{-1}$$ where the superscript $\langle -1 \rangle$ refers to the inverse under composition.
Note that this function is not, in general, defined on $H^{+}(\mathcal{B})$, but  in a uniform neighborhood of $0$.
We define the \textit{Voiculescu transform} to be the function $\varphi_{\mu}(b) := \mathcal{R}_{\mu}(b^{-1})$.
Note that $$\varphi_{\mu}^{(n)}(b) = (F_{\mu}^{(n)})^{\langle -1 \rangle}(b) - b.$$
The significance of these functions in non-commutative probability is a result of the following equalities:
$$ \mathcal{R}_{\mu \boxplus \nu} = \mathcal{R}_{\mu} + \mathcal{R}_{\nu} ; \ \varphi_{\mu \boxplus \nu} = \varphi_{\mu} + \varphi_{\nu}. $$

We refer to \cite{Speichop} for the combinatorial definition of the \textit{free cumulants}, $c_{\nu , n}^{(\ell+1)}: \otimes^{\ell} M_{n}(\mathcal{B}) \rightarrow  M_{n}(\mathcal{B})$.
The coefficients of the $\mathcal{R}$-transform are equal to $$\kappa^{\ell, n}_{\nu}(b) := c_{\nu}^{(\ell+1)}(\underbrace{b, \ldots , b}_{\ell -  times}).$$
Indeed, in an appropriate neighborhood of $0$, we have that
$$ \mathcal{R}^{(n)}_{\mu}(b) = \sum_{\ell=1}^{\infty} \kappa^{\ell, n}_{\mu}(b) \ ; \ \  \mathcal{\varphi}^{(n)}_{\mu}(b) = \sum_{\ell=1}^{\infty} \kappa^{\ell, n}_{\mu}(b^{-1}) $$
We note for the sake of clarity that the function $\mathcal{R}^{(n)}(b) b$ is sometimes used as an alternate definition of the $\mathcal{R}$-transform.
In terms of convergence of these respective series, note that
\begin{equation}\label{conv}
 \| c_{\nu, n}^{(\ell+1)}(b_{1}, b_{2} , \ldots , b_{\ell}) \| \leq (4M)^{\ell + 1} \| b_{1} \| \| b_{2} \| \cdots \| b_{\ell}\| \end{equation}
so that respective series converge provided that $\| b \| < (4M)^{(\ell + 1)/\ell}$ (re: $\| b^{-1} \| < (4M)^{(\ell + 1)/\ell})$ .

Given $\mu \in \Sigma$, we define a  linear map $\rho_{\mu}: \mathcal{B}\langle X \rangle \rightarrow \mathcal{B}$ to be the $\mathcal{B}$-bimodular linear extension of the map
$$ \rho_{\mu}(Xb_{1}X \cdots b_{\ell}X) = c^{\ell + 1}(b_{1} , b_{2} , \cdots , b_{\ell}) .$$
We define $\mathcal{B}\langle X \rangle_{0}$ to be the elements of  $\mathcal{B}\langle X \rangle$  with no constant term.
The following theorem will  arise in a key step in characterizing infinitely divisible distributions in terms of their transforms.
\begin{theorem}\cite{PV}\label{rho_pos}
Let $\mu \in \Sigma_{0}$.  The following conditions are equivalent:
\begin{enumerate}
\item The distribution $\mu$ is $\boxplus$-infinitely divisible.
\item The restriction of $\rho_{\mu}$ to  $\mathcal{B}\langle X \rangle_{0}$ is positive.
\item\label{PVremark}  There exists a self adjoint $\alpha \in \mathcal{B}$ and a $\mathbb{C}$-linear map $\sigma: \mathcal{B}\langle X \rangle \rightarrow \mathcal{B}$ satisfying \ref{CP} and \ref{bound} such that
$$ \mathcal{R}_{\mu}^{(n)}(b) = \alpha \otimes 1_{n} + \sigma(b(1_{n} - (X\otimes 1_{n}) b)^{-1}).$$
\end{enumerate}
\end{theorem}

\begin{remark}
We note here that in the $\ast$-algebra $\mathcal{B}\langle X \rangle$,  the element  $$ b(1_{n} - (X\otimes 1_{n}) b)^{-1}$$ is not , in general, an element of $\mathcal{B}\langle X \rangle$,  so that some clarification is necessary regarding domains of definition. Implicit in the proof of Theorem \ref{rho_pos} is the fact that this may be realized in a C$^{\ast}$-algebra on which the map $\sigma$ extends and, for $b \in M_{n}^{-}(\mathcal{B})$, the element 
$ b(1_{n} - (X\otimes 1_{n}) b)^{-1} \in  M_{n}^{-}(\mathcal{B})$.   This following is, therefore, an immediate consequence of Popa and Vinnikov's result (Theorem 5.10 of the aforementioned paper).
\end{remark}

\begin{proposition}\label{necessary}
Assume that $\mu \in \Sigma_{0}$ is a $\boxplus$-infinitely divisible distribution.  Then $\varphi_{\mu}$ extends to $H^{+}(\mathcal{B})$.  
\end{proposition}

The function theory for scalar-valued free probability is well developed and, in the course of the proofs of our main theorems, we will sometimes reduce operator-valued questions to this special case.
The following theorems will be of vital importance to our approach.
\begin{theorem}\cite{BV}\label{vtransform}
Let $\phi : \mathbb{C}^{+} \rightarrow \mathbb{C}^{-}$ be an analytic function.  Then $\phi$ is a continuation of $\varphi_{\mu}$
for some $\boxplus$-infinitely divisible, compactly supported  measure $\mu$ if and only if
$\mathcal{R}(z):= \phi(1/z)$ extends to a neighborhood of $0$ and 
$$ \lim_{|z| \uparrow \infty} \frac{\phi(z)}{z} = 0.$$
\end{theorem}

We note here that one of our main theorems, \ref{VT}, is a direct analogue of the previous theorem for the operator-valued setting.
We refer to \cite{NS}, Theorem 13.16, for proof of this theorem.

\begin{theorem}\label{kappa}
Let $\mathcal{R}_{\mu}(z) = \sum_{n=1}^{\infty}\kappa_{n}z^{n-1}$ for a compactly supported probability measure $\mu$.
The following are equivalent:
\begin{enumerate}
\item $\mu$ is $\boxplus$-infinitely divisible.
\item  The sequence $\{ \kappa_{n} \}_{n \geq 2}$ is positive definite.  That is, there exists a finite, positive measure $\sigma$ such that $$\kappa_{n} = \int_{\mathbb{R}} t^{n-2} d\sigma(t).$$
\item  We have that $$\mathcal{R}_{\nu}(z) = \kappa_{1} + \int_{\mathbb{R}} \frac{z}{1-tz}d\rho(t) $$
for some finite, positive measure $\rho$.
\end{enumerate}
\end{theorem}

We shall now prove some general results about the non-commutative Cauchy transform.
The converse of the following proposition is the main result of this paper.
Setting notation for the remainder of the paper, let $\sigma(b)$ denote the spectrum of this element.
We say that $|b| >C$ if $\inf_{\lambda \in \sigma(b)} |\lambda| > C$.  We state that a sequence of elements $|b_{k}| \uparrow \infty$ if for any $C > 0$ there exists a $K \in \mathbb{N}$ such that $|b_{k}| > C$ for all $k \geq K$.

\begin{proposition}\label{converse}
Let $\mu \in \Sigma_{0}$.  The Cauchy transform $G_{\mu}$ has the following properties:
\begin{enumerate}
\item\label{CT1} The Cauchy transform is non-commutative function with  $ G_{\mu} : H^{+}(\mathcal{B}) \rightarrow H^{-}(\mathcal{B})$.
\item\label{CT2} The function $h = (h^{(n)})_{n=1}^{\infty}$ where $h^{(n)}(b) := G^{(n)}_{\mu}(b^{-1})$ has uniformly analytic extension to  a neighborhood of $0$.
\item\label{CT3}  Given any sequence $\{ b_{k} \}_{k=1}^{\infty} \subset M_{n}(\mathcal{B})$ with $|b_{k}| \uparrow \infty$
we have that $b_{k}G_{\mu}(b_{k}) \rightarrow 1_{n}$ in norm.
\end{enumerate}
\end{proposition}
\begin{proof}
The non-commutative structure of the Cauchy transform was proven in \cite{V1} and  \cite{V2}.  Moreover, the proof of the remaining aspects of \ref{CT1} are contained in these references although we reprove them here for the readers convenience.

  With respect to the domain and range of $G_{\mu}$,
observe that  \ref{PV} implies that there exists a non-commutative probability space $(\mathcal{A}, E , \mathcal{B})$ and a self adjoint element $a \in \mathcal{A}$ such that  $G_{\mu}(b) = E[ (b - a\otimes 1_{n})^{-1}]$
for all $b \in M_{n}(\mathcal{B})$.  Since $ \Im{(b-a)} = \Im{(b)} > \epsilon 1_{n}$ for some $\epsilon > 0$, we have that $(b-a)^{-1} \in M_{n}^{-}(\mathcal{B})$ and the conditional expectation preserves this set.

With respect to \ref{CT2}, taking the series expansion of the multiplicative inverse, we have that $G_{\mu}(b) = \sum_{n=0}^{\infty} \mu((b^{-1}X)^{n}b^{-1})$ for $\| b^{-1} \| $ small enough.  The function $h^{(n)}(b) = \sum_{n=0}^{\infty} \mu((bX)^{n}b)$  is convergent
provided that $\| b  \|  < M$ where $M >0$ is the constant arising from \ref{bound}.

To prove \ref{CT3}, note that $|b_{k} | \uparrow \infty$ implies that $\| b_{k}^{-1}  \| \downarrow 0$. The claim follows by a cursory look at the series expansion.  
\end{proof}
Note that the scalar valued version of our main result, Theorem \ref{scalar}, has the property that it is enough to show that the asymptotic requirements are  satisfied in a single direction.  That is, one need only assume the requisite analytic properties and that $g(iy)/iy \rightarrow 1$ as $y \uparrow \infty$ in order to prove the the relevant function is a Cauchy transform.  In the operator-valued setting this is not the case as the following counterexample, due to Anshelevich and Belinschi \cite{AB}, makes clear.
\begin{example}\label{example}
Let $\mathcal{A} = M_{2}(L^{\infty}(\mathbb{R}) ) )$, $\mathcal{B} = M_{2}(\mathbb{C})$
and , if $E$ is the expected valued on $L^{\infty}(\mathbb{R}) $,
we define non-commutative probability space with conditional expectation $E \otimes 1_{2} : \mathcal{A} \rightarrow \mathcal{B}$.
Consider the non-commutative distribution generated by 
$$ a = \left( \begin{array}  {cc}
s_{1} &  0 \\
0 & s_{2} \\
\end{array} \right)
$$
where $s_{1}$ and $s_{2}$ are independent with  $(0,1)$-semicircle distribution and note that $E \otimes 1_{2}(a) = 0$.
Further, for an arbitrary element in $\mathcal{B}$,
$$ b = \left( \begin{array}  {cc}
b_{1,1} &  b_{1,2} \\
b_{2,1} & b_{2,2} \\
\end{array} \right)
$$ 
observe that 
the  distribution of $a$ has the property that \begin{equation}\label{expectation} E \otimes 1_{2}(aba) =  \left( \begin{array}  {cc}
E(b_{1,1}) &  0 \\
0 & E( b_{2,2}) \\
\end{array} \right)\end{equation}
Now, let $G_{a}$ denote the Cauchy transform of this element and consider its multiplicative inverse (the $F$-transform)
$$ F^{(n)}_{a}(b) = b - \mu(ab^{-1}a) + O( b^{-3})$$ for $b \in M_{n}^{+}(\mathcal{B}).$
Observe that, since $F_{a}$ increases the imaginary part,  the non-commutative function defined through the equalities $$ H^{(n)}(b) := b - F^{(n)}_{a}(b) : M_{n}^{+}(\mathcal{B}) \rightarrow  M_{n}^{-}(\mathcal{B}) $$ has the appropriate domain and range for a Cauchy transform.
Moreover,  a quick look at the  series expansion yields the fact that the non-commutative function defined through the equalities  $K^{(n)}(b) : = H^{(n)}(b^{-1})$ is uniformly convergent in a neighborhood zero.
Further, for any invertible matrix $b \in M_{n}(\mathcal{B})$, \ref{expectation} implies that that $bH^{(n)}(b) = b E\otimes 1_{2n} \circ E'(b^{-1}) + O(|b|^{-2})$ as $|b| \uparrow \infty$ where $E'$ is the conditional expectation from the matrix algebra onto the subalgebra of diagonal elements.  Thus, the asymptotics are correct for diagonal elements but fail for arbitrary elements of $M_{n}(\mathcal{B})$.  As this distribution fails the property that $bH^{(n)}(b) \rightarrow 1_{n}$ as $|b| \uparrow \infty$ for arbitrary $M_{n}(\mathcal{B})$, by the previous proposition, this function is not a Cauchy transform, in spite of the fact that it has the correct asymptotics for diagonal elements and all of the requisite analytic properties.
\end{example}
While the conditions on our main theorem \ref{Main_theorem} may not be weakened to this extent, it is plausible that 
asymptotic requirements  may be weakened to the assumption that $ib_{k}G^{(n)}(ib_{k}) \rightarrow 1_{n}$ as $b_{k} \uparrow \infty $
for any sequence of \textit{positive} elements $b_{k} \uparrow \infty$.  This may be the correct analogue of the weakened hypotheses in the scalar-valued case, but we are unable to prove it at this time and are not sure that it should be true.

We close this section by citing two results in function theory that are of great importance to our later proofs.  The first is the Earle-Hamilton fixed point theorem and the second is a technical estimate proven in Lemma 2.3 of \cite{BS2}.
\begin{theorem}\cite{EH}\label{EH}
Let $\mathcal{D}$ be a connected open subset of a complex Banach space $ \mathcal{X}$ and let $f$ be a holomorphic mapping of $\mathcal{D}$ into itself such that
\begin{enumerate}
\item the image $f(\mathcal{D})$  is bounded in norm;
\item  the distance between points $f(\mathcal{D}) $ and points in the exterior of $\mathcal{D}$  is bounded below by a positive constant.
\end{enumerate}
Then the mapping $ f $ has a unique fixed point $x$ in $\mathcal{D}$  and if $y$  is any point in $ \mathcal{D}$, the iterates $ f^{\circ n}(y) $ converge to $x$.
\end{theorem}

\begin{proposition}\label{Fbounded}
Assume that $\mu \in \Sigma_{0}$ with exponential bound $M$.  Then we have that that  $$ \|F_{\mu}^{(n)}(b) - b \| < 4M(1+2M/\epsilon)$$
for all $b \in M_{n}^{+,\epsilon}(\mathcal{B})$ and $n \in \mathbb{N}$.
\end{proposition}

\section{A Classification of B-Valued Cauchy Transforms.}\label{main_result}

\begin{theorem}\label{Main_theorem}
Let $g = (g^{(n)}): H^{+}(\mathcal{B}) \rightarrow H^{-}(\mathcal{B})$ denote an analytic, noncommutative function
such that the noncommutative function $h = (h^{(n)})_{n=1}^{\infty}$ defined by $h^{(n)}(b) := g^{(n)}(b^{-1})$ has
uniformly analytic extension to a neighborhood of $0$. Moreover, assume that $b_{k}g^{(n)}(b_{k}) \rightarrow 1_{n}$ in norm for any sequence $\{ b_{k} \}_{k \in \mathbb{N}} \subset M_{n}(\mathcal{B})$ with $| b_{k} | \uparrow \infty$ (in particular, $h^{(n)}(0) = 0$).
Then, $g = G_{\mu}$ for some $\mu \in \Sigma_{0}$.
\end{theorem}
The converse of this theorem is  \ref{converse}.
Before proving Theorem \ref{Main_theorem}, we begin with a lemma.

 \begin{lemma}\label{matrix_lemma}
Let $P(X) = b_{1}Xb_{2} \cdots X b_{\ell + 1}$ denote a monomial in $M_{n}(\mathcal{B})\langle X \rangle$.
Let $$\tilde{X}_{n} = \left( \begin{array}  {cc}
    X \otimes 1_{n} & 1_{n} \\
 1_{n} & X \otimes 1_{n}
   \end{array} \right) \ \in M_{2n}(\mathcal{B})\langle X \otimes 1_{2n} \rangle $$
Then, for any $k \geq \ell$ there exist matrices $c_{i} \in M_{2n}(\mathcal{B})$
for $i = 1 , \ldots , k$ such that the following equality holds:
$$ c_{1} \tilde{X}_{n}   c_{2} \cdots  \tilde{X}_{n}  c_{k + 1} = \left( \begin{array}  {cc}
    P(X) & 0 \\
 0 & 0
   \end{array} \right)$$
\end{lemma}
\begin{proof}
Let $e_{1,1} , e_{1,2} , e_{2,1} , e_{2,2} \in M_{2}(\mathbb{C})$ denote the usual matrix units.
For $i = 1 , \ldots , \ell + 1$, let $c_{i} = b_{i} \otimes e_{1,1}$ and observe that, for each $1 \leq i \leq \ell + 1$, we have that
$$ c_{1} \tilde{X}_{n} c_{2} \cdots \tilde{X}_{n} c_{i} = b_{1}(X \otimes 1_{n})b_{2} \cdots (X \otimes 1_{n})b_{i} \otimes e_{1,1} .$$
Now, observe that
$$ [ b_{1}(X \otimes 1_{n})b_{2} \cdots (X \otimes 1_{n}) b_{i} \otimes e_{1,1}] \tilde{X}_{n} [1_{n} \otimes e_{2,1}] =   b_{1}(X \otimes 1_{n})b_{2} \cdots (X \otimes 1_{n}) b_{i} \otimes e_{1,1} . $$
Letting $c_{i} = 1_{n} \otimes e_{2,1} $ for $i > \ell + 1$, our result follows by induction.
\end{proof}

We now prove our main theorem.

\begin{proof}

We begin by defining our linear map $\mu : \mathcal{B}\langle X \rangle \rightarrow \mathcal{B}$.
Let $$\mu(b_{1}Xb_{2} \cdots Xb_{\ell +1}) := \Delta_{\mathcal{R}}^{\ell + 1}h^{(1)}(\underbrace{0,\ldots,0}_{\ell + 2 \  - \ times})( b_{1} , b_{2} , \ldots , b_{\ell+1})$$ for elements $b_{1}, \ldots , b_{\ell + 1} \in \mathcal{B}.$
It follows from \cite{KVV}, theorem 3.10, that $\mu$ is a well defined, linear map on $B\langle X \rangle$.  
Consider  the operator  $\mu \otimes 1_{n}$ on $M_{n}(\mathcal{B}) \langle X \otimes 1_{n} \rangle$.
 Proposition \ref{matrix_prop} implies that 
\begin{align*}
& \Delta_{\mathcal{R}}^{\ell + 1}h^{(n)}(0, \ldots , 0 )( B_{1} , \ldots , B_{\ell + 1}) \\
& = \left(\sum_{k_{1} , \ldots , k_{\ell}= 1}^{n} \Delta_{\mathcal{R}}^{\ell}h^{(1)} (0,\ldots , 0)( b_{i,k_{1}} , b_{k_{1} , k_{2}} , \ldots , b_{k_{\ell },j}  )  \right)_{i,j=1}^{n} \\
& = \mu \otimes 1_{n} (B_{1}(X \otimes 1_{n}) B_{2} \cdots (X \otimes 1_{n}) B_{\ell + 1}).
\end{align*}

In order to prove that $ \mu \in \Sigma_{0}$, we need to show that
\begin{enumerate}[(I).]
\item $\mu(b) = b$ for all $b\in \mathcal{B}$, \label{prove1} 
\item $ \mu$ is bounded in the sense of  (\ref{bound}), \label{prove2} 
 \item $\mu$  is completely positive in the sense of (\ref{CP}). \label{prove3}
\end{enumerate}

To prove \ref{prove1}, note that $h$ is analytic at $0$
implies that $\delta h^{(1)} (0 ; b)$ is analytic in the variable $b$.
Now, pick $b \in \mathcal{B}$ invertible.  We have that
$$ \mu(b) = \Delta^{1}_{\mathcal{R}}h^{(1)}(0,0)(b) = \delta h^{(1)}(0;b) = \lim_{\zeta \rightarrow 0} \frac{h^{(1)}(\zeta b) - h^{(1)}(0)}{\zeta} = \lim_{\zeta \rightarrow 0} b[(\zeta b)^{-1}g^{(1)}((\zeta b)^{-1})]$$
By assumption, the right hand side converges to $b$ as $\zeta \rightarrow 0$ (the sequential assumption is enough since our analyticity assumption implies that a limit exists). Analytic continuation implies this holds for all $b \in \mathcal{B}$.

To prove \ref{prove2}, we note that this is equivalent to showing that 
 $$\| \mu(b_{1}Xb_{2} \cdots Xb_{\ell + 1}) \|  \leq CM^{\ell +1}$$
for a fixed $C > 0$, provided that $\| b_{1} \| = \cdots = \| b_{\ell + 1} \| = 1$.  Consider the element of $ M_{\ell + 2}(\mathcal{B})$
$$ B = \left( \begin{array} {cccccc}
0 & b_{1} & 0 & 0 & \cdots & 0 \\
0 & 0 & b_{2} & 0 & \cdots & 0 \\
0 & 0 & 0 & b_{3} & \cdots & 0 \\
\ & \vdots & \ & \ & \vdots & \ \\
0 & 0 & 0 & 0 & \cdots & b_{\ell+1} \\
0 & 0 & 0 & 0 & \cdots &0
\end{array} \right) .$$
Note that $h^{(\ell+1)}$ has a bound of $C$ on a ball of radius $r$ about $0$, independent of $\ell$ since we are assuming that $h$ is uniformly analytic.
Thus, \begin{align*}
      \| \mu(b_{1}Xb_{2} \cdots X b_{\ell + 1}) \| & = \| \delta^{\ell + 1} h^{(\ell + 2)}(0;B) \| \\
& = \| \Delta_{\mathcal{R}}^{\ell + 1} h^{(\ell + 2)} (0 , \ldots , 0)( B , \ldots , B) \| \\
& =  \| r^{-(\ell + 1)} \Delta_{\mathcal{R}}^{\ell + 1} h^{(\ell + 2)} (0 , \ldots , 0 )( r B , \ldots , r B) \| \\
& \leq C \left( \frac{1}{r} \right)^{\ell + 1} 
      \end{align*}
where the last inequality follows from Theorems \ref{Cauchy} and \ref{difference}.

To prove \ref{prove3}, we first show that, given any monomial $P(X) \in M_{n}(\mathcal{B})\langle X \otimes 1_{n} \rangle$ and self adjoint element $b_{0} \in M_{n}(\mathcal{B})$,
we have that 
\begin{equation}\label{mon_claim}
\mu \otimes 1_{n}(P(X \otimes 1_{n} +  b_{0})^{\ast}P(X \otimes 1_{n}  +  b_{0})) \geq 0
\end{equation}
 (we will then invoke Lemma \ref{matrix_lemma} to complete the proof).
Towards this end, let $P(X) = b_{1}Xb_{2} \cdots Xb_{\ell + 1}$ for $b_{1} , \ldots , b_{\ell + 1} \in M_{n}(\mathcal{B})$ and $\ell \geq 0$.
We will assume that $|b_{\ell + 1}| > c1_{n}$ for some constant $c > 0$.
Inequality \ref{mon_claim} will follow for the general case by letting $c \downarrow 0$.
Pick $\epsilon > 0$.  Let $c_{i} = \delta b_{i}$  for $i=1,\ldots , \ell$ and $c_{\ell + 1} = b_{\ell + 1}/\delta^{\ell}$ for $\delta > 0$ as yet unspecified.
Observe that $b_{1}Xb_{2} \cdots Xb_{\ell + 1} = c_{1} X c_{2} \cdots X c_{\ell + 1}$.

Consider the elements $ C , E_{0}, E_{1} \in M_{n(\ell + 1)}(\mathcal{B})$ defined as follows:
$$ C =  \left( \begin{array}  {ccccccc}
    0 & c_{1} & 0 & 0 & 0 &  \cdots & 0 \\
 c_{1}^{\ast} & 0 & c_{2} & 0 & 0 &\cdots & 0 \\
0 & c_{2}^{\ast} & 0 & c_{3} & 0 & \cdots & 0 \\
\vdots & \ & \ & \vdots & \ & \ & \vdots \\
0 & 0 & \cdots & 0 & c_{\ell - 1}^{\ast} & 0 & c_{\ell} \\
0 & 0 & \cdots & 0 & 0 & c_{\ell}^{\ast} & |c_{\ell + 1}|^{2}
   \end{array} \right)  ; \ \ E_{0} =\underbrace{1_{n} \oplus 1_{n} \oplus \cdots \oplus 1_{n}}_{\ell \ times} \oplus 0_{n}$$
and  $E_{1} = 1_{n(\ell + 1)} - E_{0}$.  Consider $b_{0} \in M_{n}(\mathcal{B}) $ as in \ref{mon_claim}.
We define a function $$\tilde{g}^{(n(\ell + 1))}(b) : = g^{(n(\ell + 1))}(b - b_{0} \otimes 1_{\ell + 1}): M_{n(\ell + 1)}(\mathcal{B})^{+} \rightarrow M_{n(\ell + 1)}(\mathcal{B})^{-}. $$
We will use these objects to prove \ref{mon_claim}, but must first single out some intermediate results.
We claim the following:
\begin{enumerate}[(i).]
\item\label{C_1}  $C + \epsilon E_{0} > \gamma 1_{n(\ell+1)}$ for some $\gamma > 0 $ provided that  $ \delta$ is small enough.
\item\label{C_2}  The $n \times n$ minor in the top left corner of $$[(C + \epsilon E_{0})(X \otimes 1_{n(\ell + 1)} + b_{0} \otimes 1_{\ell + 1})]^{2(\ell-1) } (C + \epsilon E_{0})$$ is equal to $P(X +  b_{0})P^{\ast}(X +  b_{0}) + O(\epsilon)$.
\end{enumerate}
Regarding  $\tilde{g}^{(n(\ell + 1))}$, we claim the following:
\begin{enumerate}[(a).]
\item\label{g_1} $\tilde{g}^{(n(\ell + 1))}(b) = \sum_{p=0}^{\infty} \mu( [ b^{-1}(X \otimes 1_{n(\ell + 1)} + b_{0} \otimes 1_{\ell + 1}) ]^{p} b^{-1}) $
\item\label{g_2} $b\tilde{g}^{(n(\ell + 1))}(b) \rightarrow 1_{n(\ell + 1)}$ in norm as $|b|  \uparrow \infty$
\item\label{g_3}  $\tilde{h}^{(n(\ell + 1))}(b):=\tilde{g}^{(n(\ell + 1))}(b^{-1})$ has analytic extension to a neighborhood of zero.
\end{enumerate}

To prove \ref{C_1}, pick  $\min{ \{ c^{2} , \epsilon \} } > \lambda  > 0$ (where again, $|b_{\ell + 1}| > c1_{n}$).
Let $f \in M_{n(\ell + 1)}(\mathcal{B})^{\ast}$ denote a state.
\begin{align*}
f(C + \epsilon E_{0}) &= f(E_{1}CE_{1} + \epsilon E_{0}) + f(C - E_{1}C E_{1}) \\
& \geq \lambda - \| C - E_{1}CE_{1} \| \geq \lambda - \delta \| B - E_{1}BE_{1} \|
 \end{align*}  
and the right hand side is bounded away from $0$ for $\delta$ small enough (this inequality arises since $E_{1}CE_{1} + \epsilon E_{0} \geq \lambda 1_{n(\ell + 1)}$ and $f$ is a state).  As this was an arbitrary state, our claim holds.

To prove \ref{C_2}, we set notation by considering the decomposition $M_{n(\ell + 1)}(\mathbb{C}) = M_{\ell + 1}(\mathbb{C}) \otimes M_{n}(\mathbb{C})$.
Let $\{ e_{i,j}\}_{i,j=1}^{\ell+1}$ denote the usual matrix units for $M_{\ell + 1}(\mathbb{C})$ and let $\{ F_{i,j} \}_{i,j = 1}^{\ell + 1} \subset M_{n(\ell + 1)}(\mathbb{C})$
denote the block matrix units with $F_{i,j} := e_{i,j} \otimes 1_{n}$.  Thus, we have the following equivalent formulation of \ref{C_2}:
$$F_{1,1}[(C + \epsilon E_{0})(X \otimes 1_{n(\ell + 1)} + b_{0})]^{2(\ell-1) } (C + \epsilon E_{0})F_{1,1} = [P(X + b_{0})P^{\ast}(X +  b_{0}) + o(\epsilon)] \oplus 0_{n\ell}.$$
To prove this, observe that $$ C + \epsilon E_{0} = \sum_{p = 1}^{\ell} \left[ \epsilon F_{p,p} +e_{p, p+1} \otimes c_{p} +   e_{p+1 , p}\otimes c^{\ast}_{p} \right] +   e_{\ell + 1 , \ell + 1}\otimes |c_{\ell + 1}|^{2}. $$
Pick any decomposition of $F_{1,1}$ into a product of $2\ell - 1$ matrix units:
$$ F_{1,1} = F_{1,i_{1}} F_{i_{1},i_{2}} \cdots F_{i_{2\ell - 2} , 1}$$
Observe that $(C + \epsilon E_{0})$ has non-zero entries in the block matrix units $F_{i,j}$ if and only if $i-1 \leq j \leq i+1$.
This means that such a product is non-zero and contains $F_{\ell+1 , \ell + 1}$ as one of its entires if and only if it is of the form
 $$F_{1,2}F_{2,3} \cdots F_{\ell, \ell + 1} F_{\ell + 1, \ell + 1} F_{\ell + 1, \ell} \cdots F_{2,1} = F_{1,1} . $$
The remaining non-zero terms do not contain a copy of $F_{\ell + 1 , \ell + 1 }$ so that they must contain an element of the form $F_{p,p}$ with $1 \leq p \leq \ell$, and \ref{C_2} follows immediately.

Switching to the function $\tilde{g}^{n(\ell + 1)}$, note that property \ref{g_2} follows immediately from the analogous fact for  $g^{n(\ell + 1)}$.
Property \ref{g_3} follows from the observation that $$\tilde{h}^{n(\ell + 1)}(b) = \tilde{g}^{n(\ell + 1)}(b^{-1}) = g^{n(\ell + 1)}(b^{-1} - b_{0}) = h^{n(\ell + 1)}\left( \sum_{k=0}^{\infty}[(bb_{0})^{k}b]\right) $$
which lies entirely inside the domain of $h^{n(\ell + 1)}$ for $\| b \|$ small enough.  Property \ref{g_1} follows from a basic series rearrangement argument as well as analytic continuation.

We now have the pieces in place to prove \ref{mon_claim}.
Note that \ref{C_1} implies that  $C + \epsilon E_{0}$ is  invertible so that  the map $$z \mapsto \tilde{g}^{(n(\ell + 1))}(z(C + \epsilon E_{0})^{-1})$$  sends $\mathbb{C}^{+}$ into $M_{n}(\mathcal{B})^{-}$.
Let $B_{i,j} \in M_{n}(\mathcal{B})$ for $i,j = 1 , \ldots , \ell + 1$ and consider the element $B = (B_{i,j})_{i,j = 1}^{\ell + 1} \in M_{n(\ell + 1)}(\mathcal{B})$.
Given a state $f \in M_{n}(\mathcal{B})^{\ast}$ we define a new state $$f_{1,1}(B) := f(B_{1,1}) :   M_{n(\ell + 1)}(\mathcal{B}) \rightarrow \mathbb{C}.$$
We may define a map $$G_{f,C,\epsilon}(z) = f_{1,1} \circ g^{(n(\ell + 1))}(z(C + \epsilon E_{0})^{-1}) : \mathbb{C}^{+} \rightarrow \mathbb{C}^{-}.$$
Properties \ref{g_1} and \ref{g_2} imply the following for $z \in \mathbb{C}^{+}$:
\begin{align*}
\lim_{|z| \uparrow \infty} z G_{f,C,\epsilon}(z) = & \lim_{|z| \uparrow \infty} f_{1,1} \left[ (C + \epsilon E_{0}) ( z(C + \epsilon E_{0})^{-1}  g^{(n(\ell+1))}(z(C + \epsilon E_{0})^{-1} ))  \right] \\
& = f_{1,1}( C + \epsilon E_{0} ) =  \epsilon f(1_{n}) > 0
\end{align*}
This implies that $G_{f,C,\epsilon}$ is the Cauchy transform of a finite, positive measure $\rho$ with mass equal to $\epsilon$ (see, for instance, \cite{BV} ).

Now, observe that the coefficent of  $z^{-2\ell + 1}$ for the  function $G_{f,C,\epsilon}$ is equal to $\rho(t^{2(\ell-1)}) > 0$.
Furthermore, since \begin{align*}
G_{f,C,\epsilon}(z) & = G_{\rho}(z) = \sum_{\ell = 0}^{\infty} \frac{\rho(t^{\ell})}{z^{\ell + 1}} \\
& = \sum_{\ell = 0}^{\infty} \frac{f_{1,1}(\mu([(C + \epsilon E_{0})(X \otimes 1_{n(\ell + 1)} + b_{0})]^{\ell} (C + \epsilon E_{0}) ))}{z^{\ell + 1}}
\end{align*}
we may conclude that $$f_{1,1} \circ \mu([(C+ \epsilon E_{0})(X \otimes 1_{n(\ell + 1)} + b_{0})]^{2(\ell-1)} (C + \epsilon E_{0}) )= \rho(t^{2(\ell - 1)}) > 0.$$
Recalling \ref{C_2}, it follows that $f \circ \mu ( [P(X + b_{0})P^{\ast}(X +  b_{0}) + o(\epsilon)]) > 0$.  Letting $\epsilon \downarrow 0$ and noting that $f$ was an arbitrary state, we have proven that
$$ \mu \otimes 1_{n} ( P(X + b_{0})P^{\ast}(X +  b_{0}) ) \geq 0$$
for any monomial $P(X) \in M_{n}(\mathcal{B})\langle X \rangle$.

As in section 3.5 in \cite{Speichop}, we need only show that $\mu(P(X)P^{\ast}(X)) \geq 0$ for  an arbitrary polynomial $P(X) \in M_{k}(\mathcal{B}) \langle X \rangle$ for some $k \in \mathbb{N}$ in order to complete our proof of \ref{prove3} .
  Let $P(X) = \sum_{i=1}^{N} P_{i}(X)$
with  $P_{i}(X) = b_{1}^{(i)}X b_{2}^{(i)} \cdots X b_{n(i)}^{(i)}$
and $m = \max_{i} n(i)$.
To each of these monomials $P_{i}$ we associate the family of elements $c_{1}^{(i)} , \ldots , c_{m}^{(i)} \in M_{2k}(\mathcal{B})$
that satisfy the previous lemma.  Consider the $2Nk \times 2Nk$ matrices
$$C_{j} = c_{j}^{(1)} \oplus c_{j}^{(2)} \oplus \cdots \oplus c_{j}^{(N)}$$ for each $j = 1,\ldots , m$
as well as \begin{equation}\label{tildeX}
\tilde{X} = \underbrace{ \left( \begin{array}  {cc}
    X \otimes 1_{k}  & 1 \\
 1 & X \otimes 1_{k}
   \end{array} \right) \oplus 
\left( \begin{array}  {cc}
    X \otimes 1_{k}  & 1 \\
 1 & X \otimes 1_{k} 
   \end{array} \right)
\oplus
\cdots \oplus
\left( \begin{array}  {cc}
    X \otimes 1_{k}  & 1 \\
 1 & X  \otimes 1_{k}
   \end{array} \right)}_{N \ times}
\end{equation}
By the previous lemma, we have that 
\begin{equation}\label{minor_id}
C_{1}\tilde{X} C_{2} \cdots \tilde{X}C_{m}=\left( \begin{array}  {cccccc}
   P_{1}(X) & 0 & 0 &\cdots & 0 & 0 \\
 0 & 0 & 0 & \cdots & 0 & 0  \\
0 & 0 & P_{2}(X) & \cdots  & 0 & 0  \\
 \ & \vdots  & \  & \vdots & 0 & 0 \\
0 & 0 & 0 & \cdots & P_{N}(X) & 0 \\
0 & 0 & 0& \cdots & 0 & 0
   \end{array} \right) .
\end{equation}

Consider the matrix $2Nk \times 2Nk$ matrix $$ S =   \left( \begin{array}  {cccc}
   1 & 0 & \cdots & 0 \\
 1 & 0 & \cdots & 0 \\
\ & \vdots & \ & \vdots \\
1 & 0 & \cdots & 0
   \end{array} \right)$$
We have shown that $\mu \otimes 1_{2Nk}(Q(\tilde{X}) Q(\tilde{X})^{\ast}) \geq 0$ for monomials $Q(X) \in M_{2Nk}(\mathcal{B})\langle X \rangle $. In particular, we have the following inequality: 
\begin{align*}
0 & \leq  \mu \otimes 1_{2Nk} (C_{1}\tilde{X} C_{2} \cdots \tilde{X}C_{n}SS^{\ast}(C_{1}\tilde{X} C_{2} \cdots \tilde{X}C_{m})^{\ast}) \\
& = \mu \otimes 1_{2Nk} \left[ \left( \begin{array}  {cccccc}
   P_{1}(X) P_{1}^{\ast}(X) & 0 &   P_{1}(X) P_{2}^{\ast}(X)  &\cdots &    P_{1}(X) P_{N}^{\ast}(X) & 0 \\
 0 & 0 & 0 & \cdots & 0 & 0  \\
  P_{2}(X) P_{1}^{\ast}(X) & 0 &   P_{2}(X) P_{2}^{\ast}(X)  & \cdots  &  P_{2}(X) P_{N}^{\ast}(X) & 0  \\
 \ & \vdots  & \  & \vdots & 0 & 0 \\
  P_{N}(X) P_{1}^{\ast}(X)  & 0 &   P_{N}(X) P_{2}^{\ast}(X)  & \cdots &   P_{N}(X) P_{N}^{\ast}(X)  & 0 \\
0 & 0 & 0& \cdots & 0 & 0 \\
   \end{array} \right)  \right] \\
\end{align*}
If we apply the $M_{k}(\mathcal{B})$-valued vector state generated by the $1 \times 2N$ column vectors with all entries equal to $1_{k}$,
as this is a positive map, it follows that $0 \leq \mu \otimes 1_{k} (P(X)P^{\ast}(X))$, thereby proving our theorem.
\end{proof}

We are now prepared to prove our analogue of \ref{nevanlinna1}.

\begin{corollary}\label{nevanlinna}

Let $f = (f^{(n)}): H^{+}(\mathcal{B}) \rightarrow H^{+}(\mathcal{B})$ denote an analytic, noncommutative function.
The following conditions are equivalent.
\begin{enumerate}[(a)]
\item\label{nev1}  $f = F_{\mu}$ for some $\mu \in \Sigma_{0}$.
\item\label{nev2}  The noncommutative function $k = (k^{(n)})_{n=1}^{\infty}$ defined by $k^{(n)}(b) := (f^{(n)}(b^{-1}))^{-1}$ has
uniformly analytic extension to a neighborhood of $0$. Moreover, for any sequence $\{ b_{k} \}_{k \in \mathbb{N}}$ with $|b_{k}|\uparrow \infty$, assume  that $b_{k}^{-1}f^{(n)}(b_{k}) \rightarrow 1_{n}$ in norm.
\item\label{nev3} There exists an $\alpha \in \mathcal{B}$ and a $\sigma : \mathcal{B}\langle X \rangle \rightarrow \mathcal{B}$ satisfying \ref{CP} and \ref{bound} such that, for all $n \in \mathbb{N}$, 
$$f^{(n)}(b)  = \alpha 1_{n} + b - \sigma \otimes 1_{n} (b(1-Xb)^{-1}).$$
\end{enumerate}
\end{corollary}
\begin{proof}
The equivalence of \ref{nev1} and \ref{nev2}  is an immediate consequence of the previous result.
The implication \ref{nev2} $\Rightarrow $ \ref{nev3} is Theorem $6.6$ in \cite{PV}.
Lastly, if we assume \ref{nev3}, \ref{nev2} follows immediately from the convergence properties of the series expansions which follow from the estimates in  \ref{bound}.


\end{proof}
\section{A Classification Theorem for Linearizing Transforms.}\label{transform_result}

\begin{theorem}\label{VT}
Let $\phi : H^{+}(\mathcal{B}) \rightarrow H^{-}(\mathcal{B})$ denote an analytic, non-commutative function.
Define $\mathcal{R} = (\mathcal{R}^{(n)})_{n=1}^{\infty}$ through the equations $\mathcal{R}^{(n)}(b) := \phi^{(n)}(b^{-1}) :M_{n}^{-}(\mathcal{B}) \rightarrow M_{n}^{-}(\mathcal{B})$.  The following conditions are equivalent.
\begin{enumerate}[(a)]
\item\label{tfcae1} $\phi = \varphi_{\mu}$ for some $\boxplus$-infinitely divisible $\mu \in \Sigma_{0}$.
\item\label{tfcae2} There exists a self adjoint $\alpha \in \mathcal{B}$ and a $\mathbb{C}$-linear map $\sigma: \mathcal{B} \langle X \rangle \rightarrow \mathcal{B}$ satisfying inequality \ref{bound} such that 
$$ \mathcal{R}^{(n)}(b) = \alpha \otimes 1_{n} + \sigma \otimes 1_{n} (b(1_{n} - (X\otimes 1_{n}) b)^{-1}. $$
\item\label{tfcae3} The function $\mathcal{R}$ extends to a uniformly analytic non-commutative function defined on a neighborhood of $0$ satisfying $\mathcal{R}^{(n)}(b^{\ast}) = \mathcal{R}^{(n)}(b)^{\ast}$.
Moreover, $\phi$  satisfies $b_{k}^{-1}\phi^{(n)} (b_{k})\rightarrow 0$  for any sequence with  $|b_{k}| \uparrow \infty$.
\end{enumerate}
\end{theorem}
\begin{proof}
The implication \ref{tfcae1} $\Rightarrow$ \ref{tfcae2} is theorem $5.10$ in \cite{PV}.

Regarding the implication \ref{tfcae2} $\Rightarrow $ \ref{tfcae3}, the fact that $\mathcal{R}$ respects adjoints is immediate from its definition.
The uniform analyticity in a neighborhood of $0$ follows from the series expansion
$$ \mathcal{R}^{(n)}(b) = \alpha \otimes 1_{n} + \sum_{k=0}^{\infty} \sigma(b[(X\otimes 1_{n})b]^{k}) $$
and the assumption that $\sigma$ satisfies \ref{bound}.  Regarding the claims about the function $\phi$, the fact that $\mathcal{R}$ is defined in a neighborhood of zero implies that $\phi^{(n)}$ is defined for all $b \in M_{n}(\mathcal{B})$
with $|b|$ large enough.  Moreover, extension of $\mathcal{R}$ to a neighborhood of $0$ easily implies that $b^{-1}\phi^{(n)}(b) = b^{-1}\mathcal{R}^{(n)}(b^{-1}) \rightarrow 0$ as $|b| \uparrow \infty$ since this implies that $\|b^{-1} \| \downarrow 0$.
Thus, we have shown that  \ref{tfcae2} $\Rightarrow $ \ref{tfcae3}.

Thus, our theorem reduces to the implication \ref{tfcae3} $\Rightarrow$ \ref{tfcae1}.
We begin the proof by showing that assumption \ref{tfcae3} implies that the function $\phi$ is  uniformly bounded in norm on $H^{+}_{\epsilon}(\mathcal{B})$ for every $\epsilon > 0$.
To do so, we appeal to the scalar valued case.
Pick an element $b = x + iy \in M_{n}(\mathcal{B})$ with $y > \epsilon 1_{n}$ and $x$ self-adjoint.
The uniform analyticity of the function $\mathcal{R}$ implies that  there exists $R,M > 0$ such that  $\| \phi^{(n)}(b) \|  < R$ provided that $|b| > M$.
Given a state $f \in \mathcal{B}^{\ast}$, consider the complex analytic function $$\varphi_{f,b}(z) := f \circ \phi^{(n)} \left( x + z\left( \frac{M}{\epsilon} \right) y \right).$$
Observe that $\varphi_{f,b}: \mathbb{C}^{+} \rightarrow \mathbb{C}^{-}$ and satisfies $\varphi_{f,b}(z) / z  \rightarrow 0$ as $|z| \uparrow \infty$.
By Nevanlinna's theorem, there exists a real number $\beta$ and a finite Borel measure $\rho$ such that
$$\varphi_{f,b}(z) = \beta + \int_{\mathbb{R}} \frac{1 + tz}{t-z} d\rho(t) .$$
Consider the following chain of equalities and  inequalities, note that the first equality arises easily by plugging $i$ into the  Nevanlinna representation and the last inequality follows from the fact that $\Im{(x + i (M/\epsilon)y)} > M1_{n}$:
\begin{align*} |\rho(\mathbb{R})| & = |\Im{(\varphi_{f,b}(i) )}| = \left| \Im{ \left[  f \circ \phi^{(n)} \left( x + i\left( \frac{M}{\epsilon} \right) y \right) \right] } \right|  \\  & \leq \left| f \circ \phi^{(n)} \left( x + i\left( \frac{M}{\epsilon} \right) y \right)  \right| < R
\end{align*}
Moreover, another careful look at the Nevanlinna representation and the same reasoning implies the following: $$ |\beta| =  |\Re{(\varphi_{f,b}(i) )}| < R.$$
Noting that $$ \sup_{t\in \mathbb{R}} \left| \frac{1 + tz}{t-z} \right| \leq \max \{ |z| , |1/z| \} $$
for purely imaginary $z$, 
we may conclude that 
\begin{align*} | f \circ \phi^{(n)}(b)|  & = |\varphi_{f,b}(i\epsilon/M)| \leq |\beta| +   \left| \int_{\mathbb{R}} \frac{1 + tz}{t-z} d\rho(t) \right| \\  & \leq R + \sup_{t\in \mathbb{R}} \left| \frac{1 + tz}{t-z} \right| \rho(\mathbb{R}) \leq R(1 + M/\epsilon)\end{align*}
thereby proving our claim.


With this technical fact proven, we turn to the following claims:
\begin{enumerate}[(I)]
\item\label{F1}  The map $(b + \phi^{(n)}(b))^{\langle -1 \rangle}$ exists and extends to a domain  $M_{n}^{+}(\mathcal{B}) \cup M_{n}^{-}(\mathcal{B}) \cup \{ b \in M_{n}(\mathcal{B}) :  | b |  > C \} $ for fixed $C > 0$, independent of $n$.
\item\label{F2} There exists a $\mu \in \Sigma_{0}$ such that $ (b + \phi^{(n)}(b))^{\langle -1 \rangle} = F_{\mu}^{(n)}(b)$ for all $n \in \mathbb{N}$.
\item\label{F3} The distribution $\mu$ is $\boxplus$-infinitely divisible.
\end{enumerate}

To prove \ref{F1}  we will consider the analytic map $$h_{b}(w) :=  b - \phi^{(n)}(w)$$
for each $b \in M_{n}^{+}(\mathcal{B})  \cup\{ b \in M_{n}(\mathcal{B}) :  | b |  > C \}  $ 
and appropriate $C > 0$.  We will show that this map satisfies the hypotheses of the Earle-Hamilton theorem.
Indeed, to each $b$, we will associate an open set $\Omega^{(n)}(b)$ with the property that
$h_{b}(\Omega^{(n)}(b)) \subset \Omega^{(n)}(b)$ and show that this containment is proper in the sense of $\ref{EH}$.
We will then define a function $F$ that assigns to each of these elements $b$ the fixed point of the function $h_{b}$.
We will show that a fixed point of $h_{b}$ must be the composition inverse of the map $b \mapsto b + \phi^{(n)}(b)$.

Towards this end, we must first identify the appropriate $C$ as in the statement of  \ref{F1}.   Fix $n \in \mathbb{N}$.
We claim that there exists a fixed $b_{0} \in \mathcal{B}$ such that $\phi^{(n)}(b) \rightarrow b_{0} \otimes 1_{n}$ uniformly over $n$ as $| b| \uparrow \infty$.
Indeed, analyticity of the function $\mathcal{R}$ at $0$ implies that such a limit exists. The fact that $\mathcal{R}^{(n)}(b^{\ast}) = \mathcal{R}^{(n)}(b)^{\ast}$ implies that  the limit is self adjoint.
Lastly, the fact that the transforms respect direct sums and the the fact that the zero element in $M_{n}(\mathcal{B})$ is the $n$-fold direct sum of the zero element in $\mathcal{B}$
implies that the limit is of the form $b_{0} \otimes 1_{n}$ for each $n$.

Thus, for some $\delta > 0$, there exists  $C_{0} > 0$ such that $$\| \phi^{(n)}(b) -  b_{0} \otimes 1_{n}\| < \delta$$ for $|b| > C_{0}$ (independent of $n$ as a result of uniform analyticity).
Let $C = \|b_{0} \| + C_{0} + \delta$ and assume that $|b| > C$.
Let $r = |b| - C_{0} > 0$.
We define our fixed point set as $\Omega_{b}^{(n)} = B_{r}(b) \subset M_{n}(\mathcal{B})$, the open ball of radius $r$.
Observe that for $w \in \Omega_{b}^{(n)}$, we have that
\begin{equation}\label{subtle}
\| b - h_{b}(w) \| = \| \phi^{(n)}(w) \| \leq \|b_{0} \| + \delta = C - C_{0} < |b| - C_{0} = r.
\end{equation}
Thus, $\phi^{(n)}(w) \in \Omega_{b}^{(n)}$.  Moreover, the distance to the exterior points is bounded above by $|b| - C > 0$
so that the containment is proper in the sense of \ref{EH}.  Thus, we may invoke this theorem provided that $|b| > C$.

We isolate two distinct observations for further use.
The first is that $r \uparrow \infty$ as $|b| \uparrow \infty$.
The second is that our choice of $r$ is much larger than is actually necessary.  Indeed, a close look at \ref{subtle}
allows one to conclude that $r$ need only be large enough to dominate $\|b_{0} \| + \delta$ so that one could equally well take $r$ to be uniformly bounded.  
This allows us to conclude that, if we let $F(b)$ denote the fixed point of $h_{b}$ that arises as a result of the Earle-Hamilton theorem,
we have that there exists a constant $D > 0$ such that 
\begin{equation}\label{EHbound}
\| b - F(b) \| < D
\end{equation}
provided that $|b| > C$.


We now turn our attention to those $b \in H^{+}(\mathcal{B})$.  Assume that $\Im{(b)} > \epsilon 1_{n}$ and $|b| \ngtr C$. As we have just shown, the function $\phi$ has a uniform bound of $M$
for all $b \in H^{+}_{\epsilon/2}(\mathcal{B})$.  Fix $\lambda > 0$ and let $$\Omega^{(n)}_{b} =\{ w \in M_{n}^{+}(\mathcal{B}): \ \Im{(w)}> \epsilon/2 , \|w - b \| < M + \lambda \}.$$
For $w \in \Omega^{(n)}_{b}$, observe that $\Im{(\phi^{(n)}(w))} < 0$ so that $\Im{(h_{b}(w))} > \epsilon$.
Furthermore, $$ \|b - h_{b}(w) \| = \|\phi^{(n)}(w) \| \leq M.$$  Thus, $h_{b}(\Omega^{(n)}_{b}) \subset \Omega^{(n)}_{b}$ and the distance to the exterior points is bounded by the minimum of $\lambda$ and $\epsilon/2$.

Thus, we may invoke the Earle-Hamilton theorem for all points in the statement of \ref{F1} (the case of $H^{-}(\mathcal{B})$ follows by reflexivity).
Let $F$ denote the non-commutative function with domain $ H^{+}(\mathcal{B}) \cup H^{-}(\mathcal{B}) \cup \left(  \cup_{n=1}^{\infty} \{ b \in M_{n}(\mathcal{B}) :  | b |  > C \} \right)$ 
defined by setting $F(b)$ equal to the fixed point of $h_{b}$.  The fact that this is indeed a non-commutative function follows easily from the non-commutativity of $\phi$.

We must show that $F$ is analytic.
Fix $$ b \in  M_{n}^{+}(\mathcal{B}) \cup M_{n}^{-}(\mathcal{B}) \cup  \{ b \in M_{n}(\mathcal{B}) :  | b |  > C \} .$$
Let $g_{n}(b) := h^{\circ n}_{b}(b)$ and note that it is immediate from the definition of $h_{b}$ that $g_{n}$ is analytic in $b$ and  it follows from the Earle-Hamilton theorem that $g_{n}(b) \rightarrow F(b)$ in norm.
Let $B(0,\epsilon) \subset \mathbb{C}$ denote the complex $\epsilon$-ball where $\epsilon$ is chosen so that $B_{\epsilon}(b) \subset \Omega^{(n)}(b)$. Let  $h$ denote an element in the unit ball of $M_{n}(\mathcal{B})$ and $f \in M_{n}(\mathcal{B})^{\ast}$ a state.  Consider the map 
$$f(g_{n}(b + \zeta h)):B(0, \epsilon) \rightarrow \mathbb{C}.$$ 
These are analytic functions and, by the previous arguments, bounded in norm.  Thus, by Montel's theorem, we have that
$$ f(g_{n}(b + \zeta h)) \rightarrow f(F(b + \zeta h))$$
uniformly, so that $f(F(b + \zeta h))$ is also an analytic function in $\zeta$.
By Dunford's theorem (\cite{dunford}), since $F(b + \zeta h)$ is weakly analytic in this sense, it follows that $F(b + \zeta h)$ is analytic in $\zeta$.  This implies G\^{a}teaux differentiability of our function and, since we have local boundedness, analyticity.

Observe that the equation $\phi^{(n)}(b) = \mathcal{R}^{(n)}(b^{-1})$ implies, using the Cauchy estimates in \ref{Cauchy} and the observation that $\phi^{(n)}(b) - b_{0} \otimes 1_{n} \rightarrow 0$ as $|b| \uparrow \infty$, that for $|b|$ large enough we have that
$\| \delta (\phi^{(n)} + Id)(b ; h) - h \|$ can be made arbitrarily small.  Applying the inverse function theorem an appropriate point, we have that $\phi^{(n)} + Id$ is invertible in an open subset of $M^{+}_{n}(\mathcal{B})$.
We refer  to this inverse as $\tilde{F}$.  

We claim that $\tilde{F}$ satisfies our fixed point equation.  Recall that the domains $\Omega^{(n)}(b)$ are arbitrarily large open balls as $|b| \uparrow \infty$.
Moreover, if we consider elements of the form $b' = b + \phi^{(n)}(b)$, we have that
$$ \| \tilde{F}^{(n)}(b') - b' \| = \| \phi^{(n)}(b) \| < K $$
for a fixed constant $K > 0$ and $|b|$ large enough.  As we are dealing with open maps, we can assume that $\tilde{F}^{(n)}(b) $ is defined and is an element of $\Omega^{(n)}(b)$.  As this the fixed point set of  $h_{b}$, we have that
$$ h_{b}(\tilde{F}^{(n)}(b) ) = b + \phi^{(n)}(\tilde{F}^{(n)}(b)) = b - (\tilde{F}^{(n)}(b) - \phi^{(n)}(\tilde{F}^{(n)}(b))) + \tilde{F}^{(n)}(b) = \tilde{F}^{(n)}(b).  $$
Thus, $F = \tilde{F}$ on these open sets and, by continuation,  we are left to conclude that $F =  (\phi + Id)^{\langle -1 \rangle}$ on its entire domain, proving \ref{F1}.

To prove \ref{F2}, we claim that this function $F$ satisfies the properties of \ref{nevanlinna}.  We need to show that $$h^{(n)}(b) :=  (F^{(n)}(b^{-1}))^{-1} = ((b^{-1} + \phi^{(n)}(b^{-1}))^{\langle -1 \rangle})^{-1}$$
has uniformly analytic continuation to a neighborhood of $0$. Towards this end, we consider $$(h^{(n)})^{\langle -1 \rangle} (b)= (\phi^{(n)}(b^{-1}) + b^{-1})^{-1} = [1 + b\mathcal{R}^{(n)}(b)]^{-1}b$$
where the notation is justified since we have just shown that these functions are inverses.
Also note that it is immediate from the right side of the equality that this function extends to a neighborhood of zero.

The uniform analyticity of $h^{(n)}$ may be shown using the Kantorovich theorem.  Indeed, consider the functions $k_{b}^{(n)}(b') =(h^{(n)})^{\langle -1 \rangle} (b') - b $.
Computing the derivative of this function at the origin for $h$ invertible, we have that 
\begin{align*}
\delta(k_{b}^{(n)})(0;h) & = \lim_{\zeta \rightarrow 0} \frac{k_{b}^{(n)}(\zeta h) - k_{b}^{(n)}(0)}{\zeta}
 = \lim_{\zeta \rightarrow 0} \frac{ \left( \phi^{(n)} \left(\frac{h^{-1}}{\zeta} \right) + \frac{h^{-1}}{\zeta} \right)^{-1}}{\zeta} \\
 & =  \lim_{\zeta \rightarrow 0} \left[ \zeta \phi^{(n)} \left( \frac{h^{-1}}{\zeta} \right) + h^{-1} \right]^{-1} = h
\end{align*}
where the last equality follows from the asymptotics of $\phi^{(n)}$.  Thus, through continuation, the derivative is the identity operator at the origin.

Now, observe that uniform analyticity of the function $\mathcal{R}^{(n)}$ implies that there exists $R,M > 0$ such that so
that $\mathcal{R}^{(n)}$ bounded by R, uniformly  over $n$, on a noncommutative ball of radius $M$.
Fix $M/2 > \epsilon , \delta > 0$ .  Let $b, x, y \in M_{n}(\mathcal{B})$ satisfy $\|b\| < \delta$ and $\|x \| , \| y \| < \epsilon$.
Utilizing \ref{Cauchy} as well as \ref{lipschitz}, we have the following:
$$\| \delta k_{b}^{(n)}(x , \cdot) - \delta k_{b}^{(n)}(y , \cdot) \| \leq \frac{2R \| x - y \|}{M - 2\| x - y \|} \leq \frac{4R \epsilon}{M - 4 \epsilon} = K(\epsilon). $$
Further note that
\begin{equation}\label{korigin}
k^{(n)}_{b}(0) = \lim_{\|b' \| \rightarrow 0} (b'\mathcal{R}^{(n)}(b') + 1_{n})^{-1}b' - b = -b.
\end{equation}
Utilizing the terminology from \ref{kantorovich},
let  $x_{0}= 0$. We have that the constant $$\eta = \| x_{1}(b) \| = \| k_{b}^{(n)}(0)\| = \| b \| \leq \delta.$$
Thus, the constant $h = K\eta \leq K(\epsilon)\delta \leq 1/2$ for appropriate $\epsilon , \delta$ so that the hypotheses of \ref{kantorovich} are satisfied.

Kantorovich's theorem implies that the Newton method may be utilized to provide a unique root for  $k^{(n)}_{b}$ and, moreover,
the root  lies in a fixed ball about $b$ whose radius $t^{\ast}$ and is unique in the ball of radius $t^{\ast \ast}$ with these constants defined as follows
 $$t^{\ast} = \frac{2\eta}{1 + \sqrt{ 1 - 2h}} ; \ \ t^{\ast \ast} = \frac{1 + \sqrt{1 - 2h}}{K} \geq \frac{1}{K}. $$  Note that $t^{\ast \ast}$ can be chosen to be arbitrarily large for $\epsilon$ small enough.
Also note that $t^{\ast}$ and $t^{\ast \ast}$ are independent of $n$.

Observe that $k^{(n)}_{b}(h^{(n)}(b)) = 0$.  It is not immediately clear that $h^{(n)}(b)$ is, in general, contained in the ball of radius $t^{\ast \ast}$ and this is required if we are to show that $h^{(n)}(b)$ has norm bounded by $t^{\ast} + \| b \|$.
To prove this, we consider the function $\Phi_{n} : B_{\epsilon}(0) \rightarrow B_{\epsilon + t^{\ast}}(0)$ defined implicity by $k^{(n)}(\Phi_{n}(b)) = 0$
which, as we have just shown, is well defined.  We will show that $\Phi_{n}$ is analytic and that it agrees with $h^{(n)}$ on an open set.

To prove analyticity, we define a function $\Phi_{n,k}(b) = x_{k}(b)$ where $x_{k}(b)$ is the $k$th stage of the modified Newton approximation on which the Kantorovich theorem is based.
More explicitly, we define $x_{0}(b) = 0$ and inductively define $$x_{k + 1}(b) = x_{k}(b) - (k^{(n)}_{b}(0))^{-1} \cdot k^{(n)}_{b}(x_{k}(b)) = x_{k}(b) - k^{(n)}(x_{k}(b)).$$
It is immediate from the definitions that $x_{k}(b)$ is analytic in the variable $b$.
Moreover, it is implicit in the proof of the Kantorovich theorem that the successive approximations lie in $B_{t^{\ast}}(b)$ for all $k \in \mathbb{N}$.
Thus, the functions $x_{k}(b)$ are uniformly bounded provided that $\| b\| < \epsilon$.
Thus, for a state $f \in \mathcal{B}^{\ast}$ and $h \in M_{n}(\mathcal{B})$ with $\| h \| \leq 1$, we have that the complex analytic function
$$ F_{b,k}(z) := f \circ (x_{k}( b + zh)) : B(0, \epsilon /2) \rightarrow \mathbb{C}$$
is a bounded, analytic function provided that $\| b \| < \epsilon / 2$.  Moveover, Kantorvich implies that $F_{b,k}(z) \rightarrow f \circ ( \Phi_{n}(b + zh))$ pointwise.
By Montel's theorem, we may conclude that $f \circ ( \Phi_{n}(b + zh))$ is analytic in $z$.  By Dunford's theorem, we may conclude that $\Phi_{n}(b + zh)$ is analyic in $z$.
This implies that it is G\^{a}teaux differentiable and, therefore, analytic.

With respect to the claim that $\Phi_{n}$ and $h^{(n)}$ agree on an open set, Kantorovich implies that we need only show that $h^{(n)}(b) \in B_{t^{\ast \ast}}(b)$ for $b$ in an open set.
Now, consider $b \in M_{n}(\mathcal{B})$ satisfying $|b^{-1}| > C' > C$ (where $C$ is the same constant as in \ref{subtle}) .
The inequality in \ref{EHbound} allows us to conclude that 
$$ \| h^{(n)}(b) \| = \| F^{(n)}(b^{-1})^{-1} \| \leq (C' + D )^{-1}.$$
Letting $C' \uparrow \infty$, this bound converges to $0$ and, in particular, $h^{(n)}(b) \in B_{t^{\ast \ast}}(b)$ for all invertible  $b \in B_{\gamma}(0)$ when $0 < \gamma < \epsilon / 2$ is small enough .
As $k^{(n)}(h^{(n)}(b)) = 0$, we may conclude that $\Phi_{n}(b) = h^{(n)}(b)$
so that $h^{(n)}$ extends to an analytic function on $B_{\epsilon/2}(0)$ with uniform bound $t^{\ast} + \epsilon$.  As neither $t^{\ast}$ nor $\epsilon$ depend on $n$, uniform analyticity follows.


By \ref{korigin}, we have that $k^{(n)}_{0} (0) =  0.$
 As $0$ is the root of this function and we defined $h^{(n)}$ as being equal to these roots,
it follows that $h^{(n)}(0) = 0$.  As these are open maps defined on a neighborhood of $0$, its image contains a neighborhood of $0$, so that $F^{(n)}(b) = h^{(n)}(b^{-1})^{-1}$ has the property that for all $b$ with $|b|$ is sufficiently large,
there exists $b'$ such that $F(b') = b$.  Moreover, $|b'| \uparrow \infty$  as $|b| \uparrow \infty$.
Thus, $$b^{-1}F^{(n)}(b) = [(F^{(n)})^{\langle -1 \rangle}(b')]^{-1}b' \rightarrow 1_{n}$$  as $|b| \uparrow \infty$.
This fact, combined with the previous paragraph, implies the hypotheses of \ref{nevanlinna} so that $F = F_{\mu}$ for some $\mu \in \Sigma_{0}$.

Thus, we have shown that $\phi^{(n)}(b) + b = (F_{\mu}^{(n)})^{\langle -1 \rangle} (b)$ for all $n \in \mathbb{N}$ and  $b \in H^{+}(\mathcal{B})$
for a fixed $\mu \in \Sigma_{0}$.  It remains to show \ref{F3}, namely that $\mu$ is $\boxplus$-infinitely divisible.
However, this is quite simple since, for $k\in \mathbb{N}$,  the function $\phi/k$ satisfies the hypotheses of \ref{tfcae3} in the statement of our theorem so that, as we have just shown, there exists an element $\mu_{k} \in \Sigma_{0}$ such that $\phi/k = \varphi_{\mu_{k}}$.
Therefore,
$$ \varphi_{\mu_{k}^{\boxplus k}} = k \varphi_{\mu_{k}} = \varphi_{\mu}.$$
Thus, $$\mu = \underbrace{\mu_{k} \boxplus \cdots \boxplus \mu_{k}}_{k \ times} $$
so that $\mu$ is $\boxplus$-infinitely divisible, proving \ref{tfcae1} and, therefore, our theorem.

\end{proof}

\section{Free Probabilistic Consequences.}\label{consequences}

We now prove the converse of \ref{necessary}, completing the equivalence.
Note that we assumed in Theorem \ref{VT} that the relevant non-commutative functions had extension to $H^{+}(\mathcal{B})$.
The following implies that all infinitely divisible distributions have this property so this assumption does not in any way narrow the scope of the result.

\begin{proposition}\label{extension}
Let $\mu \in \Sigma_{0}$.  Then $\mu$ is $\boxplus$-infinitely divisible if and only if  $\varphi_{\mu}$ extends to  $H^{+}(\mathcal{B})$ with range in $\overline{H^{-}(\mathcal{B})}$.
\end{proposition}
\begin{proof}
 In order to prove sufficiency, assume that  $\varphi_{\mu}$ extends to  $H^{+}(\mathcal{B})$ .
This distribution satisfies \ref{VT} \ref{tfcae3}.  Indeed, as $|b| \uparrow \infty$, we have that $b^{-1}\varphi_{\mu}(b) = b^{-1}\mathcal{R}_{\mu}(b^{-1}) \rightarrow 0$ 
since all $\mathcal{R}_{\mu}$ for $\mu \in \Sigma_{0}$ are analytic in a neighborhood of $0$.
Thus, our proposition holds.
\end{proof}

Let $CP(\mathcal{B})$ denote the set of all completely positive maps $\rho : \mathcal{B} \rightarrow \mathcal{B}$.

\begin{corollary}\label{semigroups}
Let $\mu \in \Sigma_{0}$ denote a $\boxplus$-infinitely divisible distribution.  
Then there exists a composition semigroup of distributions $\{ \mu^{\boxplus^{\rho}} \}_{ \rho \in CP(\mathcal{B})} \subset \Sigma_{0}$.  Moreover, each of these distributions is $\boxplus$-infinitely divisible.
\end{corollary}
\begin{proof}
Define the non-commutative function $\rho \circ \varphi_{\mu}: H^{+}(\mathcal{B}) \rightarrow  H^{-}(\mathcal{B}) $ by 
$(\rho \circ \varphi_{\mu})^{(n)} := \rho \otimes 1_{n} \circ \varphi_{\mu}^{(n)} : M_{n}(\mathcal{B})^{+} \rightarrow  M_{n}(\mathcal{B})^{-} $ for all $n \in \mathbb{N}$.
This new function satisfies the equivalent hypotheses of \ref{VT} (\ref{tfcae2} and \ref{tfcae3} are quite straightforward to verify) so that there exists a $\boxplus$-infinitely divisible distribution $\mu^{\boxplus \rho} \in \Sigma_{0}$
such that $\rho \circ \varphi_{\mu} =  \varphi_{\mu^{\boxplus \rho}}$, proving our result.
\end{proof}

\bibliographystyle{amsalpha}
\bibliography{AFT}
\end{document}